\documentclass [12pt,a4paper,reqno]{amsart}
\input amssymb.sty

\usepackage{latexsym,amsmath, amsfonts, graphics}
\usepackage{amsthm}
\usepackage{amssymb}
\usepackage{amsopn}
\usepackage{amscd}
\usepackage{color}

\def\eroman{\etype{\roman}}
\theoremstyle{plain}
\newtheorem{thm}{Theorem}[section]
\newtheorem{cor}[thm]{Corollary}
\newtheorem{lem}[thm]{Lemma}
\newtheorem{prop}[thm]{Proposition}
\newtheorem{definition}[thm]{Definition}
\newtheorem{theorem}[thm]{Theorem}

\theoremstyle{definition}

\newtheorem{remark}[thm]{Remark}
\newtheorem{example}[thm]{Example}
\newtheorem{rem}[thm]{Remark}
\newtheorem{defn}[thm]{Definition}

\def\Obst{\operatorname{Obst}}
\def\CAP{\operatorname{\mathcal{CAP}}}
\def\Capl{\operatorname{Cap}}
\def\Dcap{\operatorname{DCap}}
\def\Id{\operatorname{Id}}

\usepackage{amssymb}

\newcommand{\etype}[1]{\renewcommand{\labelenumi}{(#1{enumi})}}
\def\eroman{\etype{\roman}}

\def\sg{\sigma}

\def\lm{\lambda}

\def\dim{\mbox{\rm dim }}

\def\al{\alpha}
\def\MMO{\mathcal M}
\def\MM{\widehat{\mathcal M}}
\def\tMM{\tilde{\mathcal M}}

\def\l.l.o.{\it l.l.o}

\def\chiup{\raise 2pt\hbox{$\chi$}}

\DeclareMathOperator{\sgn}{sgn}

\def\a{\alpha}

\newcommand{\id}{\rm{id}}

\begin{document}

\title{The Braun-Kemer-Razmyslov Theorem for affine PI-algebras}
\author{ Alexei Kanel Belov, Louis Rowen}
\email{belova@math.biu.ac.il,  rowen@math.biu.ac.il}

\address{Department of Mathematics, Bar-Ilan University, Ramat-Gan
52900,Israel} %

 \thanks{This research of the authors was
supported by the Israel Science Foundation (grant no. 1207/12).}

 \thanks{An early version of Theorem ~\ref{Regcon} was written
   by Amitai Regev, to whom we are indebted for
 suggesting this project and providing helpful suggestions all along the way.}

\keywords{Polynomial identity, Jacobson radical,
Braun-Kemer-Razmyslov, Noetherian, Jacobson ring}

\begin{abstract}  A self-contained, combinatoric exposition is given for
the Braun-Kemer-Razmyslov Theorem over an arbitrary commutative Noetherian ring.
\end{abstract}

 \maketitle


\section {The BKR Theorem}

\subsection{Introduction}


  One of the major theorems in the theory of PI algebras is the Braun-Kemer-Razmyslov Theorem (Theorem~\ref{BKR8}
below). We preface its statement with some basic definitions.

\begin{definition}\label{affine8}
\begin{enumerate}
\item
An algebra $A$  is {\bf affine}  over a  commutative ring $C$    if $A$ is generated as an
algebra over $C$ by a finite number of elements $a_1, \dots,
a_\ell;$ in this case we write $A = C \{ a_1, \dots, a_\ell \}.$
\item
We say the algebra $A$ is  {\bf finite} if $A$ is spanned as a $C$-module by finitely many elements.

\item
  Algebras over a field are called {\bf PI algebras} if they satisfy (nontrivial) polynomial identities.

 \item
The {\bf Capelli polynomial} $\Capl_k$ of degree $2k$ is defined
as $$\Capl_k(x_1, \dots, x_k; y_1, \dots, y_k)= \sum_{\pi \in
S_k}\sgn (\pi) x_{\pi (1)}y_1 \cdots x_{\pi (k)}  y_k$$

\item $\operatorname{Jac}(A)$ denotes the Jacobson radical of the
algebra $A$ which, for PI-algebras is the intersection of the
maximal ideals of $A$, in view of Kaplansky's theorem.
 \end{enumerate}
 \end{definition}

\begin{thm} [The Braun-Kemer-Razmyslov Theorem]\label{BKR8} The Jacobson radical $\operatorname{Jac}(A)$ of any affine PI algebra $A$
over a field is nilpotent.
\end{thm}
The aim of this article is to present a readable combinatoric
proof (essentially self-contained in characteristic 0).

\medskip

 Let us put the BKR Theorem into its broader context in PI theory.
  We say a  ring   is {\bf Jacobson} if the Jacobson radical of every prime homomorphic image is
0.  For PI-rings, this means every prime ideal is the intersection
of maximal ideals.  Obviously any field is Jacobson, since its
only prime ideal 0 is maximal.  Furthermore, any commutative
affine algebra over a field is Noetherian by the Hilbert Basis
Theorem and is Jacobson, in view of
\cite[Proposition~6.37]{rowen4}, often called the ``weak
Nullstellensatz,'' implying the following two results:

\begin{itemize}
\item (cf.~Proposition~\ref{afre1}) If a commutative algebra $A$
is affine over a field, then $\operatorname{Jac}(A)$ is nilpotent.

\item (Special case of Theorem~\ref{nilrep3}) If  $A$ is a finite
algebra over an affine central subalgebra~$Z$ over a field, then
$\operatorname{Jac}(A)$ is nilpotent. (Sketch of proof: Passing to
homomorphic images modulo prime ideals, we may assume that $A$ is
prime PI, and $Z$ is an affine domain over which $A$ is
torsion-free.  The   maximal ideals of $Z$   lift up to maximal
ideals of $A$, in view of Nakayama's lemma, implying $Z \cap
\operatorname{Jac}(A)\subseteq \operatorname{Jac}(Z) = 0.$ If $0
\ne a \in \operatorname{Jac}(A),$ then writing $a$ as integral
over $Z$, we have the nonzero constant term in $Z \cap
\operatorname{Jac}(A) = 0,$ a contradiction.)
\end{itemize}

Since either of these hypotheses implies that $A$  is a
PI-algebra, it is natural to try to find an  umbrella result for
affine PI-algebras, which is precisely the   Braun-Kemer-Razmyslov
Theorem. This theorem was proved in several stages. Amitsur
~\cite[Theorem~5]{amitsur},  generalizing the weak
Nullstellensatz, proved that if $A$ is affine over a commutative
Jacobson ring, then $\operatorname{Jac}(A)$ is nil.  In
particular, $A$ is a Jacobson ring. (Later, Amitsur and Procesi
~\cite[Corollary~1.3]{amitsur.procesi} proved that
$\operatorname{Jac}(A)$ is locally nilpotent.) Thus, it remained
to prove that every nil ideal of $A$ is nilpotent.

It was soon proved that this does hold for an affine algebra which
can be embedded into a matrix algebra, see Theorem~\ref{nilrep1}
below.
 However, examples of Small~\cite{sm} showed the existence of affine PI algebras which can not be embedded into any matrix algebra.
Thus, the following theorem by Razmyslov~\cite{razmyslov3} was a
major breakthrough  in this area.

\begin{thm}[Razmyslov]\label{thm.raz}
 If
 an affine algebra  $A$ over a field  satisfies a
Capelli identity, then its Jacobson radical
$\operatorname{Jac}(A)$ is nilpotent.
\end{thm}

Although Razmyslov's theorem was given originally in
characteristic zero, he later found a proof that works in any
characteristic. As we shall see, the same ideas yield the parallel
result:

\begin{thm} \label{thm.raz1}
Let  $A$ be an affine algebra over a commutative Noetherian ring
$C$. If $A$ satisfies a Capelli identity, then any nil ideal of
$A$ is nilpotent.
\end{thm}

Following Razmyslov's theorem, Kemer~\cite{kemer.0.5} then proved

\begin{thm}~\cite{kemer.0.5}\label{kemerfirst8}
In characteristic zero, any affine PI algebra satisfies some
Capelli identity (see Theorem~\ref{kemer.capelli}).
\end{thm}
Thus, Kemer completed the proof of the following theorem:
\begin{thm}[Kemer-Razmyslov]
If  $A$ is an affine
 PI-algebra
 over a field $F$ of characteristic zero, then its  Jacobson radical $\operatorname{Jac}(A)$ is nilpotent.
\end{thm}
%

Then, using different
 methods relying on the structure of Azumaya algebras,
Braun proved the following result, which together with the
Amitsur-Procesi Theorem immediately yields Theorem~\ref{BKR8}:

\begin{thm}\label{Braun} Any nil ideal of an affine PI-algebra
over an arbitrary commutative Noetherian ring is nilpotent.
\end{thm}

Note that to prove directly that $\operatorname{Jac}(A)$ is
nilpotent it is enough to prove Theorem~\ref{Braun} and show that
$\operatorname{Jac}(A)$ is nil, which is the case case when $A$ is
Jacobson, and is called the ``weak Nullstellensatz.'' But  the
weak Nullstellensatz requires some assumption on the base ring
$C$.
 It
can be proved without undue difficulty that $A$ is Jacobson when
$C$ is Jacobson, cf.~\cite[Theorem~4.4.5]{rowen1}. Thus, in this
case    the proper general formulation of the nilpotence of
$\operatorname{Jac}(A)$ is:

\begin{thm}[Braun]\label{weaknul}
If $A$ is an affine  PI-algebra over a Jacobson Noetherian base
ring, then $\operatorname{Jac}(A)$ is nilpotent.
\end{thm}

  Small has pointed out that Theorems~\ref{Braun} and
~\ref{weaknul} actually are equivalent, in view of a trick of
\cite{RSS}. Indeed, as just pointed out, Theorem~\ref{Braun}
implies Theorem~\ref{weaknul}. Conversely, assuming
Theorem~\ref{weaknul}, one needs to show that
$\operatorname{Jac}(A)$ is nil. Modding out the nilradical, and
then passing to prime images, one may assume that $A$ is prime.
Then one embeds $A$ into the polynomial algebra $A[\lambda]$ over
the Noetherian ring $C[\lambda],$ and localizes at the monic
polynomials over $C[\lambda],$ yielding a Jacobson base ring by
\cite[Theorem~2.8]{RSS}.

\medskip
Braun's qualitative proof  was also presented in
\cite[Theorem~6.3.39]{rowen3}, and a detailed exposition, by
 L'vov~\cite{lvov}, is available in Russian. A sketch of Braun's proof  is
also given in
 \cite[Theorem~3.1.1]{amitsur.small}.

 \medskip
Meanwhile, Kemer~\cite{kemerp} proved:

\begin{thm}~\cite{kemerp}\label{kemersecond8}
If $A$ is a PI algebra (not
necessarily affine) over a field $F$ of characteristic $ p>0$, then $A$ satisfies some Capelli identity.
\end{thm}
 Together with a characteristic-free proof of Razmyslov's theorem~\ref{thm.raz}
 due to Zubrilin~\cite{zubrilin.1}, Kemer's Theorems~\ref{kemerfirst8}
and~\ref{kemersecond8} yield another proof of the
Braun-Kemer-Razmyslov Theorem~\ref{BKR8}. The paper
\cite{zubrilin.1} is given in rather general circumstances, with
some non-standard terminology. Zubrilin's method was given in
\cite{BR}, although \cite[Remark~2.50]{BR} glosses over a key point
(given here as Lemma~\ref{len.01}), so a complete combinatoric proof
had not yet appeared in print with all the details. Furthermore,
full combinatoric details were provided in \cite{BR} only in
characteristic 0 because the conclusion of the proof required
Kemer's difficult Theorem~\ref{kemersecond8}. We need the special
case, which we call ``Kemer's Capelli Theorem,'' that every affine
PI-algebra $A$ over an arbitrary field  satisfies some Capelli
identity. This can be proved in two steps: First, that $A$ satisfies
a ``sparse'' identity, and then a formal argument that every sparse
identity implies a Capelli identity. The version given here
(Theorem~\ref{sparse}) uses the representation theory of the
symmetric group $S_{n}$, and provides a reasonable quartic bound
($(p-1)p\binom{u+1}2,$ where $u = \frac {2pe (d-1)^2}3  $) for the
degree of the sparse identity of $A$ in terms of the degree $d$ of
the given PI of $A$.

It should be noted that every proof that we have cited of the
Braun-Kemer-Razmyslov Theorem ultimately utilizes an idea of
Razmyslov defining a module structure on generalized polynomials
with coefficients in the base ring, but we cannot locate full
details of its implementation anywhere in the literature. One of the
objectives of this paper is to provide these details, in
\S\ref{last.1} and \S\ref{S2.6}. Although the proof is rather
intricate for a general expository paper, we feel that the community
deserves the opportunity to see the complete argument in print.

\medskip


We emphasize the combinatoric approach here.  Aside from the
intrinsic interest in having such a proof available of this
important theorem (and characteristic-free), these methods
generalize easily to nonassociative algebras, and   we plan to use
this approach as a framework for the nonassociative PI-theory, as
initiated by Zubrilin. (The proofs are nearly the same, but the
statements are somewhat more complicated. See \cite{BBRY} for a
clarification of Zubrilin's work in the nonassociative case.)
 To keep this
exposition as readable as we can, we emphasize the case where the
base ring $C$ is a field and prove Theorem~\ref{BKR8} directly by an
induction argument without subdividing it into Theorem~\ref{Braun}
and the weak Nullstellensatz, although we also treat these general
cases.

\S\ref{Razm} follows Zubrilin's short paper~\cite{zubrilin.1}, and
gives full details of Zubrilin's proof of Razmyslov's
theorem~\ref{thm.raz}. This self-contained proof is characteristic
free.

 To complete the proof of the BKR Theorem, it remains to prove Kemer's Capelli
 Theorem.
In \S\ref{section.strong} we provide the proof in characteristic~0,
by means of Young diagrams, and \S\ref{charp} contains the
characteristic $p$ analog (for affine algebras). An alternative
proof could be had by taking the second author's  ``pumping
 procedure'' which
he developed to answer Specht's question in characteristic $p$, and
applying it to the ``identity of algebraicity''
\cite[Proposition~1.59]{BR}. We chose the representation-theoretic
approach since it might be more familiar to a wider audience. The
proof of Theorem~\ref{Braun}, over arbitrary commutative Noetherian
rings, is given in \S\ref{Noeth}.
\medskip




\subsection{Structure of the proof}\label{structure}
 We
assume that $A$ is an affine   $C$-algebra and satisfies the $n+1$
Capelli identity $\Capl_{n+1}$ (but not necessarily the
$n$~Capelli identity $\Capl_{n}$), and we induct on~$n$: if such
$A$ satisfies $\Capl_{n}$ then we assume that
$\operatorname{Jac}(A)$ is nilpotent, and we prove this for
$\Capl_{n+1}$. For the purposes of this sketch, in Steps 1 through
3 and Step 7 we assume that $C$ is a field $F$.

The same argument shows that any nil ideal  $N$ of an affine algebra
$A$ over a Noetherian ring is nilpotent, yielding
Theorem~\ref{thm.raz1}. For this result we would replace
$\operatorname{Jac}(A)$ by $N$ throughout our sketch.

We write $C\{ x,y ,t  \}$ for the \textbf{free (associative)
algebra} over the base ring $C$, with indeterminates
$x_i,y_j,t_k,z$, containing one extra indeterminate $z$ for further
use. This is a free module over~$C$, whose basis is the set of
\textbf{words}, i.e., formal strings of the letters $x_i,y_j,t_k,z$.
The $x$ and $y$ indeterminates play a special role and need to be
treated separately. We write $C\{ t \}$ for the free subalgebra
generated by the $t_k$ and $z$, omitting the $x$ and $y$
indeterminates.

\medskip
1. The induction starts with $n=1$. Then $n+1=2$, and any algebra
satisfying $\Capl_2$ is commutative. We therefore need to show
that if $A$  is commutative affine  over a field $F$,  then
$\operatorname{Jac}(A)$ is nilpotent. This classical case is
reviewed in \S\ref{rep.1}.

\medskip
2.  Next is the {\it finite} case: If $A$ is affine   over a field
$F$ and a finite module over an affine central subalgebra, then
$\operatorname{Jac}(A)$ is nilpotent. This case was known well
before Razmyslov's Theorem, and is  reviewed in \S\ref{rep.11}.
Theorem~\ref{thm.raz1} follows whenever $A$ is a subring of a
finite dimensional algebra over a field.

\medskip
3. Let $\CAP_n=T(\Capl_n)$ be the $T$-ideal generated by
$\Capl_n$, and let $\mathcal{CAP}_n(A)\subseteq A$ be the ideal
generated in $A$ by the evaluations of $\Capl_n$ on $A$, so
$A/{\mathcal{CAP}_n(A)}$ satisfies $\Capl_n$. Therefore, by
induction on $n$, $\operatorname{Jac}(A/{\mathcal{CAP}_n(A)})$ is
nilpotent. Hence there exists $q$ such that
$$
\operatorname{Jac}(A)^q\subseteq
\mathcal{CAP}_n(A),\quad\mbox{so}\quad
\operatorname{Jac}(A)^{2q}\subseteq (\mathcal{CAP}_n(A))^2.
$$

\medskip

4.   ~In \S\ref{rep.2} we work over an arbitrary base ring $C$
(which need not even be Noetherian), and for any algebra $A$
introduce the ideal $I_{n,A} \subset A[\xi_{n,A}]$, for commuting
indeterminates~$\xi_{n,A}$, which provides ``generic'' integrality
relations for elements of $A$. Let $\widehat {C\{ x,y ,t \} }:=
C\{x,y,t\}/\mathcal{CAP}_{n+1},$ the relatively free algebra of
$\Capl_{n+1}$.  Taking the ``doubly alternating'' polynomial
$$f= t_1 \Capl_n( x_1,\ldots, x_n)t_2 \Capl_n( y_1,\ldots, y_n) t_3,$$
we construct, in Section~\ref{module.M(g)}, the key $ \widehat {C\{
 t  \}}$-module $\MM \subset \widehat {C\{ x,y ,t  \}}$, which
contains the polynomial $\hat f$. A combinatoric argument given in
Proposition~\ref{sof.1} applied to $\widehat {C\{ x,y ,t  \}}$
(together with substitutions) shows that $I_{n,\widehat {C\{ x,y ,t
\}}}\cdot \MM=0$.
\medskip

5.  ~We introduce the \textbf{obstruction to integrality}
$\Obst_n(A)= A\cap I_{n,A} \subset A$  and
 show that
$A/{\Obst_n(A)}$ can be embedded into a  finite algebra over an
affine central $F$-subalgebra; hence
$\operatorname{Jac}(A/{\Obst_n(A)})$ is nilpotent. This implies
that there exists $m$ such that
$$
\operatorname{Jac}(A)^m\subseteq \Obst_n(A).
$$
The proof of this step applies Shirshov's Height Theorem
\cite{shirshov2}, \cite[Theorem~2.3]{BR}.

\medskip
6. We   prove that $\Obst_n(A)\cdot (\mathcal{CAP}_n(A))^2=0$ over
an arbitrary ring $C$. This is obtained from Step 4 via a
sophisticated specialization argument involving \textbf{free
products}.

\medskip
7. We put the pieces together. When $C$ is a field, Step 3 shows
that $\operatorname{Jac}(A)^{q}\subseteq \mathcal{CAP}_n(A)$ for
some $q$, and Step~5 shows that
 $ \operatorname{Jac}(A)^{m}\subseteq \Obst_n(A)$ for some $m$. Hence
$$
\operatorname{Jac}(A)^{2q+m}\subseteq \Obst_n(A)\cdot
(\mathcal{CAP}_n(A))^2=0,
$$
which completes the proof of Theorem~\ref{thm.raz}. When $C$ is
Noetherian, any nil ideal $N$ of $C$ is nilpotent, so the
analogous argument shows that
 $ N^{m}\subseteq \Obst_n(A)$ for some $m$. Hence
$$
N^{2q+m}\subseteq \Obst_n(A)\cdot (\mathcal{CAP}_n(A))^2=0,
$$
proving Theorem~\ref{thm.raz1}.
\medskip

\subsection{Special cases}

We need some classical special cases.

\subsubsection{The commutative case}\label{rep.1}$ $


Our main objective is to prove that the Jacobson radical
$\operatorname{Jac}(A)$ of an affine   PI-algebra~$A$ (over a
field) is nilpotent. We start with the classical case for which
$A$ is commutative.

\begin{rem}\label{quest.1} Any commutative affine algebra $A$ over a Noetherian base ring $C$ is
Noetherian, by Hilbert's Basis Theorem, and hence the intersection
of its prime ideals is nilpotent,
cf.~\cite[Theorem~16.24]{rowen5}.

But for any ideal $I \triangleleft A,$ the algebra $A/I$ is also
Noetherian, so the intersection of the prime ideals of $A$
containing $I$ is nilpotent modulo $I$.

\medskip

\end{rem}

%
%
%

\begin{prop}\label{afre1}  If $H$ is a commutative affine
algebra over a field, then  $\operatorname{Jac}(H)$  is
nilpotent.\end{prop}
\begin{proof}
 The ``weak Nullstellensatz''  \cite[Proposition~6.37]{rowen4}
says that $H$ is Jacobson, and thus the Jacobson radical
$\operatorname{Jac}(H)$ is contained in the intersection of the
prime ideals of $H$. But any commutative affine algebra is
Noetherian, so we conclude with Remark~\ref{quest.1}.
\end{proof}

\subsubsection{The finite case}\label{rep.11}$ $
To extend this to noncommutative algebras, we start with some other
classical results:
\begin{enumerate}
\item \cite[Theorem~15.23]{rowen5} (Wedderburn) Any nil subring of
an $n\times n$  matrix algebra over a field  is nilpotent, of
nilpotence index $\le n$ (in view of \cite[Lemma~15.22]{rowen5}).

\item \cite[Theorem~15.18]{rowen5} (Jacobson) The Jacobson radical
of an $n$-dimensional algebra over a field  is nilpotent, and thus
has nilpotence index $\le n,$ by (1).

\item Any algebra finite over a Noetherian central subring $C$, is
Noetherian (This follows at once from induction applied to
\cite[Proposition~7.5]{rowen4}.

\end{enumerate}

\begin{theorem} \label{nilrep1} Suppose $A= C\{ a_1, \dots, a_\ell\}$ is an  affine algebra
  over a commutative  Noetherian
ring $C$, with $A \subseteq M_n(K)$ for a suitable commutative
$C$-algebra $K$. Then

\begin{enumerate}
\item Any nil   subalgebra $N$ of $A$ is nilpotent, of bounded
nilpotence index $\le mn$, where $m$ is given in the proof. When
$K$ is reduced, i.e., without nonzero nilpotent elements, then
$m=1$, so $N^n = 0.$

\item If $C$ is a field, then $\operatorname{Jac}(A)$ is
nilpotent.
\end{enumerate}
\end{theorem}
\begin{proof} For each  $1 \le k \le \ell,$ write each $a_k$ as an $n\times n$ matrix $(a_{ij}^{(k)}),$ for $a_{ij}^{(k)} \in K$,
and let $H$ be the commutative $C$-subalgebra of $K$ generated by
these finitely many $a_{ij}^{(k)}$; then $H$ is $C$-affine. We can view each $a_k $ in $M_n(H)$, so  $A \subseteq
M_n(H)$.

\medskip
(1) Let $N\subseteq A$ be a nil subalgebra. Now $A\subseteq M_n(H)$, so $N\subseteq M_n(H)$ and is nil.
Let $P\subseteq H$ be prime. The homomorphism $H\to H/P$ extends to
$$
M_n(H)\to M_n(H/P)~~(\cong M_n(H)/M_n(P)).
$$
Let $\bar N$ be the image of $N$, so $\bar N =(N+M_n(P))/M_n(P)$ so
$\bar N \subseteq   M_n(H)/  M_n( P) \cong M_n(H/P)\subseteq M_n(L)$
where $L$ is the field of fractions of the domain $H/P$. By
Wedderburn's theorem $\bar N^n=0$ which implies that $N^n\subseteq
M_n(P)$ (since $P=0$ in $H/P$ and in~$L$).
%
 Hence, letting $U$ denote
the prime radical of $H$, we have $N^n  \subseteq M_n(U)$. But, in
view of Remark~\ref{quest.1}, we have $U ^m = 0$ for some~$m.$ (If
$K$ is reduced then $U=0$, implying $m=1.$) We conclude that
$$N^{mn} = (N^n)^m  \subseteq (M_n(U))^m =M_n(U^m)= 0.$$

\medskip

 (2)
 We need here the well-known fact \cite[Exercise~15.28]{rowen5} that  when  $ J \triangleleft A,$ with  $J$  nilpotent, then
$ \operatorname{Jac}(A/J)= \operatorname{Jac}(A)/J$. It follows at
once that if $ \operatorname{Jac}(A/J)$  is nilpotent, then
$\operatorname{Jac}(A) $ is nilpotent.

 \medskip
 By hypothesis $H$ is affine over the field $C$, so $\operatorname{Jac}(H)$ is nilpotent,
 and thus
  $M_n(\operatorname{Jac}(H)) = \operatorname{Jac}(M_n(H))$ is
  nilpotent.
Let $\tilde A=A/(A\cap M_n(\operatorname{Jac}(H)))$ and $\tilde
H=H/\operatorname{Jac}(H)$. Then
$$
\tilde A \subseteq M_n(H/\operatorname{Jac}(H))=M_n(\tilde H),
$$
and $\operatorname{Jac}(\tilde H)=0$. Thus we may assume that
$\operatorname{Jac}(H) = 0,$ and we shall prove that $J^n = 0,$
where $J = \operatorname{Jac}(A).$

\medskip
For any maximal ideal $P$ of $H$, we see that $H/P$ is an affine
field extension of $C$, and thus is finite dimensional ~over $C$,
by  \cite[Theorem~5.11]{rowen4}. But then the image of $A$ in
$M_n(H/P)$ is finite dimensional over~$C$, so the image $\bar J$
of $J$ is nilpotent,   implying $\bar J^n  = 0$. Hence   $J^n$ is
contained in $\cap M_n(P) = M_n(\cap P) = 0,$ where $P$ runs over
the maximal ideals of~$H$.
\end{proof}

\begin{theorem} \label{nilrep3} Suppose $A$ is an algebra that is finite
 over $C$, itself an affine algebra  over a field. Then
$\operatorname{Jac}(A)$ is nilpotent.\end{theorem}
\begin{proof} Since $A$ is Noetherian, its nilradical $N$ is nilpotent by \cite[Remark~16.30(ii)]{rowen5}, so modding
out $N$ we may assume that $A$ is semiprime, and thus the
subdirect product of prime algebras $\{ A_ i = A/P_i: i \in I \}$
finite over their centers.  If $\operatorname{Jac}(A)^n \subset
P_i$ for each $i\in I$,   then $\operatorname{Jac}(A)^n \subset
\cap  P_i = 0.$

So
we may assume that $A$ is prime. But localizing over the center,
we may assume that $C$ is a field. Let $n = \dim _C A.$ Then $A$ is
embedded via the regular representation into $n \times n$ matrices
over a field, and we are done by Theorem~\ref{nilrep1}.
\end{proof}

Since not every affine PI-algebra might satisfy the hypotheses of
Theorem~\ref{nilrep1},  cf.~\cite{sm} and \cite{lew}, we must
proceed further.

\section{Proof of Razmyslov's Theorem}\label{Razm}

In this section we give full details for Zubrilin's proof of
Theorem~\ref{thm.raz}.

%
%
%
%
%
%
%
%


\subsection{Zubrilin's approach} $ $
\subsubsection{The operator $\delta_{k,z}^{(x,n)}$}\label{Zu1}$ $

Let us fix notation for the next few sections. $C$ is an arbitrary
commutative ring. We start with a
  polynomial $f: = f(x_1,\ldots,
x_{n})\in C\{ x, y,  t\}$ in $\vec x = \{x_1,\ldots, x_{n }\}$
(which we always notate), as well as possibly $\vec y =
\{y_1,\ldots, y_{n }\}$ (which we sometimes notate), and
   $\vec t = \{t_1, \dots\}$, all of which are  noncommutative indeterminates.

\begin{defn}\label{D2.1}
Let $f(\vec x , \vec y,  \vec t)$ be multilinear in the $x_i$ (and
perhaps involving additional indeterminates $\vec{y}$ and
$\vec{t}$).   Take $0\le k\le n$, and expand
$$
f^*=f((z+1)x_1,\ldots,(z+1)x_n,\vec y, \vec t),
$$
where  $z$ is a new noncommutative indeterminate. Then we write
$$\delta^{(x,n)}_{k,z}(f) := \delta^{(x,n)}_{k,z}(f)(x_1,\ldots,
x_{n},z)$$ for the homogeneous component of $f^*$ of degree $k$ in
the noncommutative indeterminate $z$. (We have suppressed
 $\vec{y}$, $\vec{t}$ in the notation, as indicated above.)

 For example let $n=2$ and $f=x_1x_2$. Then
$$
(z+1)x_1(z+1)x_2=zx_1zx_2+zx_1x_2+x_1zx_2+x_1x_2.
$$
Hence $ \delta^{(x,2)}_{0,z}(x_1x_2)=x_1x_2,$ $~
\delta^{(x,2)}_{1,z}(x_1x_2)=zx_1x_2+x_1zx_2,$ and  $~
\delta^{(x,2)}_{2,z}(x_1x_2)=zx_1zx_2$.

\medskip
More generally, for any $h \in C\{ t  \}$ we write
$\delta^{(x,n)}_{k,h}(f) := \delta^{(x,n)}_{k,h}(f)(x_1,\ldots,
x_{n},h)$, i.e., the specialization of $\delta^{(x,n)}_{k,z}(f) $
under $z \mapsto h.$
\end{defn}

\begin{remark}\label{finish.1}$ $

1. In calculating $\delta^{(x,n)} _{k,z}(f)$, the substitution
$x_i\to (z+1)x_i$ is applied to the first $n$ positions in $f$ but
not to the other positions. For example, the last (i.e.~$n+1$ st)
variable in $f(x_1,\ldots,x_{n-1},x_{n+1},x_n)$ is $x_n$, not
$x_{n+1}$.
Hence,
to calculate $\delta_{k,z}^{(x,n)}
(f(x_1,\ldots,x_{n-1},x_{n+1},x_n))$ we apply $x_i\to (z+1)x_i$ to
all $x_i$'s except $x_n$.

\medskip
2.~We can also write
$$
\delta_{k,z}^{(x,n)}(f(x_1,\ldots,x_n,\vec t))=\sum_{1\le
i_1<\cdots <i_k\le n}f(x_1,\ldots,x_{n }, \vec t)\mid_{x_{i_j}\to
z{x_{i_j}}}=~~~~~~~~~~~~~~~~~~~~~~~~~~
$$
$$~~~~~~~~~~~~~~~~~~~~~~~~~~~~~~~~~
=\sum_{1\le i_1<\cdots <i_k\le n}f(x_1,\ldots,zx_{i_1},\ldots zx_{i_k} ,\ldots,x_{n}, \vec t).
$$

3. ~In case $f=f(x_1,\ldots,x_{n },y_1,\ldots, y_{n })$ also
involves indeterminates $y_1,\ldots, y_{n }$,  we still have
$$
\delta_{k,z}^{(x,n)}(f)=\sum_{1\le i_1<\cdots <i_k\le n}f\mid_{x_{i_j}\to z{x_{i_j}}},
$$
indicating that the other indeterminates $y_1,\ldots, y_{n }$
remain fixed. Analogously,
$$
\delta_{k,z}^{(y,n)}(f)=\sum_{1\le i_1<\cdots <i_k\le n}f\mid_{y_{i_j}\to z{y_{i_j}}},
$$
and the indeterminates $x_1,\ldots, x_{n }$ are fixed.

\end{remark}

\begin{defn}\label{alt1}
A   polynomial $f(x_1,\ldots, x_{n },\vec t\;)$   is
\textbf{alternating} in $x_1,\dots,x_n$ if  $f$ is multilinear in
the $x_i$ and
\begin{equation}\label{alt10} f(x_1,\ldots, x_i, \ldots, x_j,  \ldots, x_{n },\vec t\;)+f(x_1,\ldots, x_j, \ldots, x_i,  \ldots, x_{n },\vec t\;)
=0
 \text{ for all } i <j.
 \end{equation}
\end{defn}

A stronger definition, which would suffice for our purposes,  is
to require that
 \begin{equation}\label{alt2} f(x_1,\ldots, x_i, \ldots, x_i,  \ldots, x_{n },\vec t\;) =0;\end{equation}  i.e., we get 0 when specializing $x_j$ to $x_i$
 for any   $1 \le i <j \le n.$
We get  \eqref{alt1}  by linearizing~\eqref{alt2}, and can recover
\eqref{alt2} from \eqref{alt1} in   characteristic $\ne 2$.)

\begin{lem}\label{alter.1}
Let $f(x_1,\dots,x_n,\vec{t}\;)$ be multilinear and alternating in
$x_1,\dots,x_n$. 
 Then for
each~$0\le k\le n$,
$\delta_{k,z}^{(x,n)}(f(x_1,\dots,x_n,\vec{t}\;))$ is also
alternating  in $x_1,\ldots,x_n$.
\end{lem}

\begin{proof}
Let $v=1+\varepsilon z$ where $\varepsilon$ is a central
indeterminate. Obviously $f(vx_1,\dots,vx_n,\vec{t}\;)$ is also
alternating in $x_1,\dots,x_n$. Since
$$
f(vx_1,\dots,vx_n,\vec{t}\;)=\sum_{k=0}^n
\left(\delta_{k,z}^{(x,n)}(f(x_1,\dots,x_n,\vec{t}\;)\right)\cdot\varepsilon^k
$$
is alternating in $x_1,\dots,x_n,$ it follows that each
$\delta_{k,z}^{(x,n)}(f(x_1,\dots,x_n,\vec{t})$ is alternating in
$x_1,\dots,x_n$.
\end{proof}

\begin{rem}\label{extr} $ $

1.  
 Since
$\mathcal{CAP}_n$ is generated as a $T$-ideal by polynomials
alternating in $x_1, \dots, x_n$, we have
$$\delta_{k,z}^{(x,n)}(\mathcal{CAP}_n) \subseteq \mathcal{CAP}_n
\qquad \text{and} \qquad \delta_{k,z}^{(x,n)}(\mathcal{CAP}_{n+1})
\subseteq \mathcal{CAP}_{n+1}.$$

2. The results proved for the indeterminate $z$ specialize to an
arbitrary polynomial $h$, and thus can be formulated for  $h$.

\begin{lem}\label{extr1} The  $\delta^{(x,n)} _{k,z}(f)$-operator is functorial, in the sense that if $\vec a = (a_1, \dots, a_m) \in A$ and $h(\vec a) = h'(\vec a)$, then
$\delta^{(x,n)} _{k,h}(f)(\vec a) = \delta^{(x,n)} _{k,h'}(f)(\vec
a).$
\end{lem}
\begin{proof}  We get the same
result in Definition~\ref{D2.1} by specializing $z$   to $h$ and then to $ \vec a$, as we
get by specializing $z$   to  $ h',$ and then to $ \vec a$.
\end{proof}

This observation is needed in our later
specialization arguments.

\end{rem}

 The
following observation, which is rather well known, motivates
Proposition~\ref{len.2} below.
Let $V=Cx_1\oplus \cdots \oplus Cx_n$ and let $z:V\to V$ be a
linear transformation from $V$ to $V$. Let
$$
\det(\lm I-z)=\sum_{k=0}^n c_k(z)\lm^k
$$
be the characteristic (``Cayley-Hamilton") polynomial of $z$. Then we have
the following formula from \cite[Theorem~1.4.12]{rowen1}:
$$
\delta_{k,z}^{(x,n)}(\Capl_n(x_1,\ldots,x_n;\vec{y}))= \sum_{1\le
i_1<\cdots<i_k\le n}
\Capl_n(x_1,\ldots,zx_{i_1},\ldots,zx_{i_k},\ldots,x_n;\vec{y})=
~~~~~~~~~~~~~~~~~~~~~~~~~~~~~~~~~~~~
$$
$$~~~~~~~~~~~~~~~~~~~~~~~~~~~~~~~~~~
~~~~~~~~~~~~~~~~~~~~~~~~~~~~~~~~~~~~~~~~~~=c_k(z)\cdot
\Capl_n(x_1,\ldots,x_n;\vec{y} ),
$$
and the coefficients $c_k(z)$ are independent of the particular
indeterminates $x_1,\ldots,x_n$. Proposition~ \ref{len.2} below displays a
similar phenomenon.

\subsection{Zubrilin's Proposition}\label{rep.21}

 Our goal in this  section is
Proposition~\ref{sof.1}. Let us define the terms used there.

Let $\widehat { C\{ x,y ,t  \}}$ denote the relatively free
algebra $C \{ x,y,t \}/\mathcal{CAP} _{n+1}$.  We denote   the
image of   $f\in C \{x,y,t \}$ in  $\widehat {C\{ x,y ,t \}}$
by~$\hat f .$

\begin{rem}\label{Noeth} If $A$ satisfies $\Capl _{n+1}$, then any
algebra homomorphism $\varphi:  C\{ x,y ,t  \}\to A$ naturally
induces an algebra homomorphism $\widehat{\varphi}: \widehat{ C\{
x,y ,t \}}\to A$ given by $$\widehat{\varphi}(\hat{f} )=
\varphi({f}),$$ since $\mathcal{CAP} _{n+1}\subseteq \ker
\varphi.$
\end{rem}

\begin{remark}\label{finish.4}
Let $f(x_1,\ldots, x_{n+1})$ be~multilinear in $x_1,\ldots, x_{n+1}$ and alternating in $x_1,\ldots, x_n$. 
Construct
\begin{eqnarray}\label{zubrilin3}
\tilde f =\tilde f(x_1,\ldots,x_{n+1})=\sum_{k=1}^{n+1} (-1)^{k-1}
f(x_1,\ldots,x_{k-1},x_{k+1},\ldots,x_{n+1},x_k).
\end{eqnarray}

(All other variables occurring in $f$ are left untouched.)

Then $\tilde f$ is $(n+1)$-alternating in $x_1,\ldots,x_{n+1}$.
\end{remark}

\begin{prop}\label{zubrilin4}
Let $f(x_1,\ldots,x_n, x_{n+1})$ be multilinear in
$x_1,\ldots,x_n, x_{n+1}$ and alternating in $x_1,\ldots,x_n$ (so
$\tilde f$ of Equation~\eqref{zubrilin3} is $(n+1)$-alternating).
Then
$$
\sum_{j=0}^n (-1)^j
\delta_{j,z}^{(x,n)}(f(x_1,\ldots,x_n,z^{n-j}x_{n+1}))\equiv 0
\quad \text{modulo } \mathcal{CAP}_{n+1}.
$$
\end{prop}

\begin{proof} 
Throughout we work modulo $\mathcal{CAP}_{n+1}$. Since $\tilde f$
is $(n+1)$-alternating, we have
$$ 0 \equiv
\tilde
f=f(x_2,x_3\ldots,x_{n+1},x_1)-f(x_1,x_3\ldots,x_{n+1},x_2)+\cdots+(-1)^nf(x_1,x_2,\ldots, x_n,x_{n+1}) .$$
Thus, modulo $\mathcal{CAP}_{n+1}$ the last summand
$(-1)^nf(x_1,x_2,\ldots, x_n,x_{n+1})$ can be replaced by minus
the sum of the other summands:
$$
(-1)^nf(x_1,x_2,\ldots, x_n,x_{n+1})\equiv  \sum_{k=1}^{n}(-1)^k
f(x_1,\ldots, x_{k-1},  x_{k+1},\ldots, x_n,x_{n+1},x_k) ,
$$
 Given $0\le j\le
n$, substitute $x_{n+1}\mapsto z^{n-j}x_{n+1}$, so
$$
(-1)^nf(x_1,x_2,\ldots, x_n,z^{n-j}x_{n+1})\equiv\sum_{k=1}^{n}
(-1)^kf(x_1,\ldots, x_{k-1},  x_{k+1},\ldots.
x_n,z^{n-j}x_{n+1},x_k).
$$
Applying $\delta_{j,z}^{(x,n)}$ and summing with sign, we get
$$
(-1)^n\sum_{j=0}^n (-1)^j \delta_{j,z}^{(x,n)}(f(x_1,\ldots,x_n,
z^{n-j}x_{n+1})\equiv~~~~~~~~~~~~~~~~~~~~~~~~~~~~~~~~~~~~~~~~~~~~~~~~~~~~~~~~~~
$$
$$
\equiv  \sum_{j=0}^n (-1)^j
\sum_{k=1}^{n}(-1)^k\delta_{j,z}^{(x,n)} (f(x_1,\ldots, x_{k-1},
x_{k+1},\ldots, x_n,z^{n-j}x_{n+1},x_k))=
$$
$$~~~~~~~~~~~~~~~~~~~~~~~~~~~~~~~~~~~~~~~~~~~~~~~~~~~~~
=\sum_{k=1}^{n} (-1)^k \sum_{j=0}^n (-1)^j\delta_{j,z}^{(x,n)}
(f(x_1,\ldots, x_{k-1},  x_{k+1},\ldots, x_n,z^{n-j}x_{n+1},x_k)).
$$
 Denote $
g_{j,k}=f(x_1,\ldots, x_{k-1},  x_{k+1},\ldots,
x_n,z^{n-j}x_{n+1},x_k), $ and
\begin{eqnarray}\label{zubrilin1.1}
Q_k=\sum_{j=0}^n (-1)^j\delta^{(x,n)}_{j,z}( g_{j,k}).
\end{eqnarray}

It
 suffices to 
show that $Q_k \equiv 0$ for each $k$. Note that  in calculating
$\delta^{(x,n)}_{j,z}( g_{j,k})=\delta_{j,z}^{(x,n)}
(f(x_1,\ldots, x_{k-1}, x_{k+1},\ldots, x_n,z^{n-j}x_{n+1},x_k))$,
$x_k$ is unchanged (since it is the last indeterminate), while for
all other $x_i$'s (in particular -- for $x_{n+1}$) we substitute
$x_i\mapsto (z+1)x_i$, cf.~ Remark~\ref{finish.1}.1. Therefore
$$
\delta_{j,z}^{(x,n)}(g_{j,k})=\delta_{j,z,[k']}^{(x,n)}(g_{j,k})+\delta_{j,z,[k'']}^{(x,n)}(g_{j,k})
$$
where

\medskip

$\delta_{j,z ,[k']}^{(x,n)}(g_{j,k})$ is the sum of the monomials
of $\delta_{j, z }^{(x,n)}(g_{j,k})$  having $z$-degree $j$, where
$x_{n+1}$ was replaced by $z x_{n+1}$;

\medskip
and

\medskip

$\delta_{j, z,[k''] }^{(x,n)}(g_{j,k})$ is the sum of the
monomials of $\delta_{j, z }^{(x,n)}(g_{j,k})$  having $z$-degree
$j$, where $x_{n+1}$ was unchanged.

\medskip
It is not difficult to see that   for $j>0$,
$$
\delta_{j,z,[k']}^{{(x,n)}} f(x_1,\ldots, x_{k-1}, x_{k+1},\ldots,
x_n,z^{n-j}x_{n+1},x_k)= $$ $$  \qquad \quad
\delta_{j-1,{z},[k'']}^{(x,n)} f(x_1,\ldots, x_{k-1},
x_{k+1},\ldots, x_n,z^{n-j+1}x_{n+1},x_k),
$$
namely
$$
\delta_{j,z,[k']}^{{(x,n)}}(g_{j,k})=
\delta_{j-1,{z},[k'']}^{(x,n)}(g_{j-1,k}).
$$

It also follows from the definitions that
$\delta_{0,z,[k']}^{{(x,n)}}(g_{0,k})=
\delta_{n,z,[k'']}^{{(x,n)}}(g_{n,k})=0$. Hence
$$
\sum_{j=0}^n
(-1)^j\delta_{j,z,[k']}^{{(x,n)}}(g_{j,k})=\sum_{j=1}^n
(-1)^j\delta_{j,z,[k']}^{{(x,n)}}(g_{j,k})=
$$
$$
=\sum_{j=1}^n
(-1)^j\delta_{j-1,{z},[k'']}^{(x,n)}(g_{j-1,k})=-\sum_{j=0}^{n-1}
(-1)^j\delta_{j,{z},[k'']}^{(x,n)}(g_{j,k}),
$$
and
$$\sum_{j=0}^n (-1)^j\delta_{j,z,[k'']}^{{(x,n)}}(g_{j,k})=\sum_{j=0}^{n-1} (-1)^j\delta_{j,z,[k'']}^{{(x,n)}}(g_{j,k}).$$

\medskip
Summing in~\eqref{zubrilin1.1} we get
$$
Q_k=\sum_{j=0}^n (-1)^j\delta^{(x,n)}_{j,z}( g_{j,k})\equiv
\sum_{j=0}^n (-1)^j\left(\delta_{j,z,[k']}^{{(x,n)}}g_{j,k}+
\delta_{j,z,[k'']}^{(x,n)}g_{j,k}\right)\equiv 0.
$$
\end{proof}

\subsubsection{ The module $\MM$ over the relatively free algebra of $\Capl_{n+1}$}\label{module.M(g)}
$ $
%

We need a special sort of alternating polynomials.

\begin{defn} A polynomial $f(x_1,\ldots, x_{n};y_1,\ldots ,y_n;\vec t)$, where
$\vec t$ denotes other possible indeterminates, is \textbf{doubly
alternating} if $f$ is linear and alternating in $x_1,\ldots,x_n$
and $y_1,\ldots,y_n$.
\end{defn}

 Our main example is the
\textbf{double Capelli} polynomial \begin{equation}\label{put1}
\Dcap_n = t_1 \Capl _n(x_1,\ldots, x_{n}; \vec t)t_2
\Capl_n(y_1,\ldots ,y_n; \vec {t'})t_3. \end{equation}

Here $\vec t$ and $ \vec {t'}$ are arbitrary sets of extra
indeterminates. We suppress the indeterminates $\vec t$,  $\vec
{t'}$, and $t_1,t_2,t_3$ from the notation, since we do not
   alter them.

\begin{defn}\label{module.1} Let $\MMO$ denote the $C$-submodule of $C\{ x,y ,t  \}$
consisting of all doubly alternating polynomials (in $x_1
\ldots,x_n,$ and in $y_1,\ldots,y_n$).

  $\MM$ denotes the image of  $\MMO$ in $\widehat {C\{ x,y ,t  \}}$, i.e.,  the $C$-submodule of $\widehat {C\{ x,y ,t  \}}$
consisting of the images of all doubly alternating polynomials (in
$x_1 \ldots,x_n,$ and in $y_1,\ldots,y_n$).
\end{defn}

\begin{remark}\label{finish.3}
  $\MM$ is a $\widehat {C\{ t  \}}$-submodule of $\widehat {C\{ x,y ,t  \}}$,
  namely $\widehat {C\{ t  \}}\MM\subseteq \MM$. Indeed,
  let $h\in C\{ t\}$ and $f\in \mathcal M$. If either $h$ or $f$ is in ${\mathcal{CAP}}_{n+1}$ then
  $hf\in {\mathcal{CAP}}_{n+1}$;  hence the product $\hat h\hat f=\widehat{hf}$ is well defined.
  Moreover, if $f=f(x_1,\ldots,x_n,y_1,\ldots, y_n,{\vec t}\;)$ is doubly alternating in the $x$'s
  and in the $y$'s, and   $h\in C\{{\vec t}\;\}$, then $hf$ is
  doubly alternating in the $x$'s and in the~ $y$'s.

\end{remark}
\subsubsection{The Zubrilin action}$ $

The theory hinges on the following amazing result, which we prove
 in
Section~\ref{ster} below. (This is also proved in
\cite[Theorem~4.82]{BR}, but more details are given here.)

\begin{prop}\label{len.2}
Let  $f(x_1,\ldots, x_{n};y_1,\ldots ,y_n)$ be  doubly alternating
in $x_1,\ldots,x_n$  and in $y_1,\ldots,y_n$  (perhaps involving
additional indeterminates). Then for any polynomial $h$,
\begin{equation}  \delta_{k,h}^{(x,n)}(f)\equiv\delta_{k,h}^{(y,n)}(f)\quad
modulo ~\mathcal{CAP} _{n+1};\end{equation}
namely,
$$
\sum_{1\le i_1<\cdots <i_k\le n}f\mid_{x_{i_j}\to
h{x_{i_j}}}\equiv \sum_{1\le i_1<\cdots <i_k\le
n}f\mid_{y_{i_j}\to h{y_{i_j}}} ~ modulo ~\mathcal{CAP} _{n+1}.
$$
%

\end{prop}

Before proving Proposition~\ref{len.2} we deduce some of its
consequences.

 \begin{rem}\label{welld}
It follows from Proposition~\ref{len.2} that
$\delta_{k,h}^{(x,n)}(f)- \delta_{k,h}^{(y,n)}(f)\in
{\mathcal{CAP}}_{n+1} $ whenever $f\in \MM$, so working modulo
${\mathcal{CAP}}_{n+1}$
  we  can suppress $x$ in the notation,
  writing
$\hat  \delta_{k,h}^{(n)}(\hat f)$ for~$\hat
\delta_{k,h}^{{(x,n)}}(\hat f) $.
\end{rem}

\subsubsection{Commutativity of the operators
$\delta_{k,h_j}^{(n)}$  modulo $\mathcal{CAP}
_{n+1}$}\label{finish.2}$ $

We use $\MM$ instead of $\MMO$ because of the following lemma.

\begin{lem}\label{module.10} (i) $\delta_{k,h}^{(n)}$ induces a
well-defined map $\hat \delta_{k,h}^{(n)}: \MM \to \MM $ given by
$\hat{\delta}_{k,h}^{(n)}(\hat f)=\widehat{\delta_{k,h}^{(x,n)}(
f)}$.

(ii) $\hat \delta_{k,h}^{(n)}$  produces the same result using the
indeterminates $x$ or $y$.
\end{lem}
\begin{proof} (i) If $f(x_1,\ldots,x_n,y_1,\ldots,y_n)$ and $g(x_1,\ldots,x_n,y_1,\ldots,y_n)$
are doubly alternating polynomials, with $\hat f = \hat g$, then
$f-g \in \mathcal{CAP}_{n+1},$ so by Remark~\ref{extr}(1),
$\delta^{(n)}_{k,h}(f-g)\in\mathcal{CAP}_{n+1}$ and hence
$\widehat{\delta_{k,h}^{(n)}( f-g)}=0.$ Therefore we have

$$0 =\widehat{\delta_{k,h}^{(n)}(
f-g)}= \widehat{\delta_{k,h}^{(n)}( f)}-
\widehat{\delta_{k,h}^{(n)}( g)}= \hat{\delta}_{k,h}^{(n)}(\hat
f)-\hat{\delta}_{k,h}^{(n)}(\hat g),
$$ proving that $\hat
\delta_{k,h}^{(n)}$ is well-defined.

(ii) The assertion follows from Remark~\ref{welld}, which shows
that $\widehat{\delta_{k,h}^{(x,n)}(
f)}=\widehat{\delta_{k,h}^{(y,n)}(f)}.$
\end{proof}

\medskip

\begin{lem}\label{len.01}
Let $f=f(x_1,\ldots,x_n;y_1,\ldots,y_n)$ be doubly alternating in
$x_1,\ldots,x_n$ and in $y_1,\ldots,y_n$ (and perhaps involving
other indeterminates). Let $1\le k,\ell\le n$. Then for any $h_1,
h_2 \in C\{ t  \},$
\begin{eqnarray}\label{last.11}
\hat \delta_{k,h_1}^{(n)}\hat \delta_{\ell,h_2}^{(n)}(\hat f)=
\hat \delta_{\ell,h_2}^{(n)}\hat \delta_{k,h_1}^{(n)}(\hat f) .
\end{eqnarray}
\end{lem}

\begin{proof}
Equation~ \eqref{last.11} claims that modulo ~$\mathcal{CAP}
_{n+1}$,
\begin{enumerate}\eroman
\item
$$
\delta_{k,h_1}^{(x,n)}\delta_{\ell,h_2}^{(x,n)}(f)\equiv
\delta_{\ell,h_2}^{(x,n)}\delta_{k,h_1}^{(x,n)}(f)\quad\mbox{and}\quad
$$
\item
$$
\delta_{k,h_1}^{(x,n)}\delta_{\ell,h_2}^{(y,n)}(f)\equiv
\delta_{\ell,h_2}^{(y,n)}\delta_{k,h_1}^{(x,n)}(f)\quad\mbox{and}\quad
$$
\item
$$
\delta_{k,h_1}^{(y,n)}\delta_{\ell,h_2}^{(y,n)}(f)\equiv
\delta_{\ell,h_2}^{(y,n)}\delta_{k,h_1}^{(y,n)}(f).
$$
\end{enumerate}
The middle equivalence (ii) is an obvious equality. The first and
third  equivalences are similar, and we prove the first.
 By Proposition~\ref{len.2},  by (ii), and again by Proposition~\ref{len.2},
 modulo $\mathcal{CAP} _{n+1}$  we can write
  $$
\delta_{k,h_1}^{(x,n)}\delta_{\ell,h_2}^{(x,n)}(f)\equiv
\delta_{k,h_1}^{(x,n)}\delta_{\ell,h_2}^{(y,n)}(f)\equiv
\delta_{\ell,h_2}^{(y,n)}\delta_{k,h_1}^{(x,n)}(f)\equiv
\delta_{\ell,h_2}^{(x,n)}\delta_{k,h_1}^{(x,n)}(f).
$$
 Note that in the last step, Lemma~\ref{alter.1} was applied (to $\delta_{k,h_1}^{(x,n)}(f)$).
 \end{proof}



\subsubsection{The ideal $I_{n,A} \subset A[\xi_{n,A}] $ and the
annihilator of $\hat M$}\label{rep.2}$ $

\begin{defn}\label{obst.1}
For each $a\in A$ let $\xi _{1,a},\ldots, \xi _{n,a}$ be  $n$
corresponding new commuting variables, and   construct $
A[\xi_{n,A}] =A[\xi _{1,a},\ldots, \xi _{n,a}\mid a\in A]$. Let
$I_{n,A}\subseteq  A[\xi_{n,A}] $ be the ideal generated in
$A[\xi_{n,A}]$ by the elements
$$
a^n+\xi _{1,a}a^{n-1}+\cdots +\xi _{n,a},\quad a\in A,
$$
namely
$$
I_{n,A}=\langle a^n+\xi _{1,a}a^{n-1}+\cdots +\xi _{n,a}\mid a\in A\rangle.
$$

\end{defn}

\begin{rem}\label{obst1}
In view of Proposition~\ref{len.2}, the map $\hat
\delta_{k,h}^{(n)}: \MM \to \MM $ of Lemma~\ref{module.10} yields
an action of the $C$-algebra $\widehat {C\{t \}}[\xi_{n,\widehat {C\{t \}}}]$ on $\MM$, given by $\xi_{k,h}f=\delta_{k,h}^{(n)}(f)$.
\end{rem}

 Working with the relatively free algebra, our next goal is to prove  that
$I_{n,\widehat {C\{t \}}}\cdot \MM=0$. For that we shall need the
next result.

\begin{prop}(Zubrilin)\label{zubrilin2}
Assume that a multilinear polynomial $g(x_1,\ldots,x_n)$
 is alternating in
$x_1,\dots,x_n$.  Then, modulo $\mathcal {CAP}_{n+1}$, 
$$
\sum_{k=0}^n (-1)^kh^{n-k}\delta_{k,h}^{(n)}(g)\equiv 0
$$
for any $h\in  {C\{t\}}$. In particular, if $g$ is doubly
alternating, then (again $\text{modulo } \mathcal{CAP}_{n+1}$)
$$
\sum_{k=0}^n (-1)^kh^{n-k}\delta_{k,h}^{(n)}(g)\equiv 0.
$$
\end{prop}
\begin{proof} First we take $h$ to be an indeterminate $z$.
Let $f(x_1,\ldots,x_{n+1})=x_{n+1}g(x_1,\ldots,x_n)$. By Proposition~\ref{zubrilin4},
 $\widehat {C\{ x,y ,t  \}}$ satisfies the identity
$$
\sum_{j=0}^n (-1)^j \delta_{j,z}^{(n)}(f(x_1,\ldots,x_n,z^{n-j}x_{n+1}))\equiv 0.
$$
Note that in computing
$\delta_{j,z}^{(n)}(f(x_1,\ldots,x_n,z^{n-j}x_{n+1}))$, the last
indeterminate is $x_{n+1}$ and is unchanged,
cf.~Remark~\ref{finish.1}.1, so
$$
\delta_{j,z}^{(n)}(f(x_1,\ldots,x_n,z^{n-j}x_{n+1}))=z^{n-j}x_{n+1}\delta_{j,z}^{(n)}g(x_1,\ldots,x_n).
$$
Using Proposition~\ref{zubrilin4}, we have that
$$
\sum_{j=0}^n (-1)^j
z^{n-j}x_{n+1}\delta_{j,z}^{(n)}(g(x_1,\ldots,x_n))\in \mathcal{CAP}_{n+1}.
$$
The proof now follows by substituting $x_{n+1}\mapsto 1$ and
  $z\to h\in {C\{t\}}$.

\end{proof}

As a consequence we can now prove the key result:

\begin{prop}\label{sof.1}
Let $\MM$ be the module given by Definition~\ref{module.1}. Then,
 $I_{n,\widehat {C\{  t
\}}}\cdot \MM= 0$.

\end{prop}
\begin{proof}
We prove that $ I_{n,\widehat {C\{  t  \}}}\cdot \MM = 0,$ by
showing for any doubly alternating polynomial $f(x_1,\ldots,
x_n,y_1,\ldots, y_n)\in \MM$   and $h\in \widehat {C\{  t \}}$,
that
$$
(h^n+\xi_{1,h}h^{n-1}+\cdots +\xi_{n,h})f\equiv 0 \pmod
{\mathcal{CAP}_{n+1}}.
$$
It follows from the action $\xi_{k,h}f=\delta_{k,h}^{(n)}(f)$ and
from Proposition~\ref{zubrilin2} that $\text{modulo }
\mathcal{CAP}_{n+1}, $
$$
(h^n+\xi_{1,h}h^{n-1}+\cdots +\xi_{n,h})f=\sum_{k=0}^n
(-1)^kh^{n-k}\delta_{k,h}^{(n)}(f)\equiv 0.
$$
\end{proof}

\subsection{The ideal  $\Obst_n(A)\subseteq A$}\label{rep.22}

In order to utilize these results about integrality, we need another
concept. We define $\Obst_n(A)=A\cap I_{n,A}$, viewing $A \subset
A[\xi_{n,A}]$.
\begin{remark}$ $
\begin{enumerate}
\item
Let
\begin{eqnarray}\label{A.hat}
\bar  A=  A[\xi_{n,A}] /I_{n,A},
\end{eqnarray}
with $f: A[\xi_{n,A}]\to \bar  A$ the natural homomorphism, and
$f_r:A\to \bar A$   be the restriction of $f$ to $A$. Then
$$\ker(f_r)=A\cap I_{n,A}=\Obst_n(A).$$
 \item
 Note that for every $a\in A$, $f(a)$ is $n$-integral (i.e.,~integral of degree $n$) over $ C[\bar \xi _{i,A}],$
and thus over the
 center of $\bar A$.
Indeed, apply the homomorphism $f$ to the element
$$
a^n+\xi _{1,a}a^{n-1}+\cdots +\xi _{n,a}~~(\in I_{n,A})
$$
to get
$$
{\bar  a}^n+{\bar \xi }_{1,a}{\bar  a}^{n-1}+\cdots +\bar \xi _{n,a}=  (a^n+\xi _{1,a}a^{n-1}+\cdots +\xi _{n,a})+I_{n,A} =0.
$$
\end{enumerate}
\end{remark}
\begin{lem}\label{gcond}
$\ker(f_r)$ also is the intersection of all kernels $\ker (g)$ of
the following maps $g$:

\medskip
 $g:A\to B$,
where $B$ is a $C$-algebra, and $g:A\to B$ is a homomorphism
 such that for any $a\in A$, $g(a)$ is $n$-integral over the center of $B$.
\end{lem}
\begin{proof}
Denote the above intersection $\cap_g\ker(g)$ as $\Obst'_n(A)$.
Then $\Obst'_n(A)\subseteq \Obst_n(A)$ since $ker(f_r)$ is among
these $\ker (g)$. To show the opposite inclusion we prove

\medskip
Claim: For such $g:A\to B$, $\ker (g)\supseteq A\cap
I_{n,A}=\Obst_n(A)$.

\medskip

 Extend $g$ to $g^*: A[\xi_{n,A}] \to B$ as follows:
$g^*(a)=a $ if $a\in A$, while $g^*(\xi _{i,a})=\beta_ {i,a}$. We claim that $g^*(I_{n,A})=0$. Indeed,
let
$$
r=a^n+\xi _{1,a}a^{n-1}+\cdots +\xi _{n,a}
$$
be one of the  generators of $I_{n,A}$.

By assumption there exist $\beta_ {1,a},\ldots,\beta_ {n,a}$ in the center of $B$ satisfying
\begin{eqnarray}\label{mozart}
g(a)^n+\beta_ {1,a}g(a)^{n-1}+\cdots+\beta_ {n,a}=0.
\end{eqnarray}
Hence,
$$
g^*(r)=g(a)^n+\beta_ {1,a}g(a)^{n-1}+\cdots+\beta_ {n,a}=0.
$$
This shows that as claimed, $g^*(I_{n,A})=0$.

\medskip
Finally, if $a\in A\cap I_{n,A}$ then $g(a)=g^*(a)=0$. Hence $a\in
\ker(g)$, so $\ker(g)\supseteq A\cap I_{n,A}=\Obst_n(A)$.
\end{proof}

\begin{cor}\label{mozart.2}
If every $a\in A$ is $n$-integral (over the base field), then
$\Obst_n(A)=0$.
\end{cor}
\begin{proof}
The assumption implies that in the above, the identity map
$id=g:A\to A$ satisfies the condition of Lemma~\ref {gcond}.
Hence
$0=\ker(g)\supseteq \Obst_n(A)$, and the proof follows.
\end{proof}

This corollary explains the notation $\Obst_n(A)$: it is the {\bf
obstruction} for each $a\in A$ to be $n$-integral. The next result
technically is not needed, but helps to show how $\Obst$ behaves.

\begin{lem}\label{mozart.3}
$\Obst_{n-1}(A)\supseteq \Obst_n(A)$.
\end{lem}
\begin{proof}
Represent
$$
\Obst_{n-1}(A)=\cap_h \ker(h)\qquad\mbox{and} \qquad
\Obst_{n}(A)=\cap_g \ker(g),
$$
with the respective conditions of $n-1$ integrality and of $n$
integrality. Take $a\in A$ and $h:A\to B$ with $h(a)$ being $n-1$
integral over the center of $B$. Then
 $h(a)$ is also  $n$ integral over the center of $B$.
Hence every $\ker (h)$ in $\Obst _{n-1}(A)$ also appears in the
intersection $\Obst_{n}(A)=\cap_g \ker(g)$, and the assertion follows.
\end{proof}

\subsection{Reduction to finite modules}

The reduction to finite modules is done using Shirshov's theorem.

\begin{prop}\label{Ra1} Let $A=C\{a_1,\ldots,a_\ell\}$ have PI degree $d$ over the base ring $C$.
Then the affine algebra $ A/\Obst_n(A)$ can be embedded in an
algebra which is finite over a  central affine subalgebra.

\end{prop}
\begin{proof}

Let $B\subseteq A$ be the subset of the words in the alphabet
$a_1,\ldots,a_\ell$ of length $\le d$. By Shirshov's Height Theorem there exists an integer $h$ such that the set
$$
W=\{b_1^{k_1}\cdots b_h^{k_h}\mid b_i\in B,\quad\mbox{any $k_i\ge 0$}\}
$$
spans $A$ over the base ring $C$.

Similarly to $A[\xi_{n,A}]$, construct $A[\xi_{n,B}]\subseteq
A[\xi_{n,A}]$:
$$
A[\xi_{n,B}]= A[\xi _{1,b},\ldots, \xi _{n,b}\mid b\in B],
$$
and let $I_{n,B}$ be the   ideal
$$
I_{n,B}=\langle b^n+\xi _{1,b}b^{n-1}+\cdots +\xi _{n,b}\mid b\in
B\rangle\triangleleft A[\xi_{n,B}]
$$

Denote
\begin{eqnarray}\label{A.prime}
A'=A[\xi_{n,B}]/I_{n,B}.
\end{eqnarray}
We show that $A'$ is finite over an affine central subalgebra and
thus is Noetherian.

\medskip
Given $a\in A$, denote $ a'=a+I_{n,B}\in A'$, and similarly
$\xi'_{i,b}=\xi_{i,b}+I_{n,B}$. Then for every $b\in B$, $ b'$ is
$n$-integral over $C[\xi'_{n,B}]$, where
$$
C[\xi'_{n,B}]=C[\xi'_{1,b},\ldots,\xi'_{n,b}\mid b\in
B]\subseteq\mbox{center}(A').
$$
Hence the finite subset
\begin{eqnarray}\label{prime1}
 W'=\{ {b'}_1^{k_1}\cdots  {b'}_h^{k_h}\mid b_i\in B,\quad\mbox{ $k_i\le n-1$}\}\quad(\subseteq A')
\end{eqnarray}
spans $A'$ over $C[\xi'_{n,B}]$. Thus $A'$ is finite
 over the affine central subalgebra
$C[\xi'_{n,B}]\subseteq\mbox{center}(A')$  and thus is Noetherian.

\medskip

Restricting the natural map
$g:A[\xi_{n,B}]\to A'=A[\xi_{n,B}]/{I_{n,B}}$
to $A$, we have
\begin{eqnarray}\label{justg}~~~~~~~~~~~~~~~~~~~~~~~
g_r:A\to A' \quad\mbox{$(a\mapsto a'=a+I_{n,B})$}
\end{eqnarray}
which satisfies
\begin{eqnarray}\label{mozart.4}
\ker (g_r)=A\cap I_{n,B}\subseteq A\cap I_{n,A}=\Obst_n(A).
\end{eqnarray}

\medskip
Let
\begin{eqnarray}\label{A.tilde}
\tilde A=A/\Obst_n(A),
\end{eqnarray}
and for $a\in A$ denote $\tilde a=a+\Obst_n(A)\in \tilde A$.
 We then have the corresponding subset $\tilde
B=\{\tilde b\mid b\in B\}\subseteq \tilde A$, as well as the set
of commutative variables $\xi_{n,\tilde B}$ and the ideal $\tilde
I_{n,\tilde B}$. Let $A^*=\tilde A[\xi_{n,\tilde B}]/{\tilde I_{n,\tilde B}}$.

\medskip
Replacing $A$ by $\tilde A$ and $\xi_{i,B}$ by $\xi_{i,\tilde B}$, we clearly have the natural homomorphism
$$
\tilde g:\tilde A[\xi_{n,\tilde B}]\to \tilde A[\xi_{n,\tilde B}]/{\tilde I_{n,\tilde B}}:=A^*,
$$
with restriction
$$\tilde g |_{\tilde A}=\tilde g_r: \tilde A\to A^*.
$$

\medskip
Note that each $\tilde a\in \tilde A$ is $n$-integral over the
center of $\tilde A$, implying, by Corollary~\ref{mozart.2}, that
$ \Obst_n(\tilde A)=0$.
Then, as in~\eqref{mozart.4}, we deduce that
$$
\ker (\tilde g_r)\subseteq \Obst_n(\tilde A)~~(=0).
$$
Hence $\tilde g_r$  embeds $\tilde A= A/\Obst_n(A)$ into
$A^*$.
Note that $A^*$ is  a finite algebra over the affine central
subalgebra  $Q\subseteq A^*$ generated by the  finitely many
central elements   $\xi_{k,\tilde B}+\tilde I_{n,\tilde B}$.


  Denote $b^*=\tilde b +I_{n,\tilde B}$. Then, as in~\eqref{prime1}, the finite subset
$$
 W^*=\{ {b_1^*}^{k_1}\cdots  {b_h^*}^{k_h}\mid b_i\in B,\ \mbox{ $k_i\le n-1$}\}\quad(\subseteq A^*)
$$
spans $A^*$ over $Q$.
\end{proof}

\subsection{Proving that
$\Obst_n(A)\cdot (\mathcal{CAP}_n(A))^2=0$}\label{last.1}$ $

 In this
 section we show how Proposition~\ref{sof.1} implies   that $\Obst_n(A)\cdot
(\mathcal{CAP}_n(A))^2=0$, thereby completing the proof of
Razmyslov's Theorem. For this, we need to specialize down to given
algebra~$A$, requiring a new construction, the \textbf{relatively
free product}, which enables us to handle  $A$ together with
polynomials. Since, to our knowledge, this crucial step, which is
needed one way or another in every published proof of the BKR
theorem, has not yet appeared in print in full detail, we present
two proofs, one faster but more ad hoc (since we intersect with $A$
and bypass certain difficulties), and the second more structural.

Both approaches are taken in the context of varieties in universal
algebra, by taking the free product of $A$ with the free associative
algebra, and then modding out the identities defining its variety.

\subsubsection{The relatively  free product}\label{rep.230}$ $


\begin{defn}      The \textbf{free
product}  $A *_C B$ of $C$-algebras $A$ and $B$ is their coproduct
in the category of algebras. \end{defn}

  (For $C$-algebras with 1, there are
canonical $C$-module maps $$A \to A \otimes 1 \subset A \otimes _C
B, \qquad B \to 1 \otimes B \subset A \otimes _C B,$$ viewed
naturally as  $C$-modules, so $A *_C B$ can be identified with the
tensor algebra of $A \otimes _C B$, as 
reviewed
in \cite[Example~18.38]{rowen5}.)

Although the results through Theorem~\ref{last.3} hold over any commutative
base ring $C$, it is easier to visualize the situation for algebras
 over a field $F$, in which case we have an explicit
description of $A[\xi_{n,A}] *F\{ x;y;t\}$:

Fix
 a   base $\mathcal B_A = \{ 1 \} \cup \mathcal B_0 $ of $A $ over $F$, and
let $\mathcal B$ be the monomials in the  $\{\xi_{n,a}: a \in A\}$
with coefficients in $\mathcal B_A$. (For algebras without 1, we
take  $\mathcal B =\mathcal B_0 $.) Thus $\mathcal B$ is an
$F$-base of $A[\xi_{n,A}]$, and $A[\xi_{n,A}] *F\{ x;y;t\}$ is the
vector space having base comprised of all elements of the form
$b_0 h_1 b_1 h_2 b_2 \cdots h_m b_m$ where $m \ge 0,$ $b_0, b_m
\in \mathcal B,$  $b_1, \dots, b_{m-1}\in \mathcal B \setminus \{
1 \}$, and the $h_i$ are nontrivial words in the indeterminates
$x_i,y_j,t_k.$ The free product
 $A[\xi_{n,A}] *F\{ x;y;t\}$ becomes an algebra via juxtaposition of
 terms. In other words, given $$g_j = b_{j,0} h_{j,1} b_{j,1} h_{j,2} b_{j,2} \cdots h_{j,m_j}
 b_{j,m_j}$$
 for $j = 1,2$, we write $b_{1,m_1}b_{2,0} = \a _1 + \sum _k \a_k b_k$ for $\a _k \in F$ and
 $b_k$ ranging over $\mathcal B \setminus \{ 1 \}$, and
 define
\begin{equation}\label{mul}\begin{aligned} g_1 g_2 = & \a _1 b_{1,0} h_{1,1} b_{1,1} h_{1,2}
b_{1,2} \cdots (h_{1,m_1}  h_{2,1}) b_{2,1} h_{2,2} b_{2,2} \cdots
h_{2,m_2} \\ & + \sum _k \a _k b_{1,0} h_{1,1} b_{1,1} h_{1,2}
b_{1,2} \cdots h_{1,m_1} b_k h_{2,1} b_{2,1} h_{2,2} b_{2,2}
\cdots h_{2,m_2} .\end{aligned}\end{equation}
 For example, if $b_1 b_2 = 1+ b_3 + b_4 ,$ then
 $$(b_2 h_{1,1}b_1)(b_2  h_{2,1} b_2) =  b_2 (h_{1,1}h_{2,1}) b_2 + b_2 h_{1,1}b_3  h_{2,1}
 b_2 + b_2 h_{1,1}b_4  h_{2,1} b_2 .$$

\subsubsection{The relatively  free product of $A $ and $C\{ x; y; t\}$ modulo a T-ideal}\label{rep.23}
$ $

Even for algebras over an arbitrary base ring $C$,  we can
describe the free product of a $C$-algebra with $C\{ x;y;t\}$ by
going over the same construction and mimicking the tensor product.
Namely we form the free $C$-module $M$ having base comprised of
all elements of the form $a_0 h_1 a_1 h_2 a_2 \cdots h_m a_m$, $
h_1 a_1 h_2 a_2 \cdots h_m a_m$, $a_0 h_1 a_1 h_2 a_2 \cdots h_m
$, and $ h_1 a_1 h_2 a_2 \cdots h_m  $ where $m \ge 0,$
  $a_0, \dots, a_{m}\in A$, and the
$h_i$ are nontrivial words in the indeterminates $x_i,y_j,t_k.$

The free product $A *C\{ x;y;t\}$ is $M/N$, where $N$ is the
submodule generated by all $$a_0 h_1 a_1 h_2\cdots a_i \cdots h_m
a_m + a_0 h_1 a_1 h_2\cdots a_i' \cdots h_m a_m - a_0 h_1 a_1
h_2\cdots (a_i+a_i') \cdots h_m a_m,$$
$$(ch_1)-ch_1, $$ $$ c a_0 h_1 a_1 h_2\cdots a_i
 - a_0
h_1 a_1 h_2\cdots (ca_i) \cdots h_m a_m,\qquad a_i \in A, \ c\in
C;$$
 $A *C\{ x;y;t\}$ becomes an algebra via juxtaposition of
 terms, i.e., given $$g_j = a_{j,0} h_{j,1} a_{j,1} h_{j,2} a_{j,2} \cdots h_{j,m_j}
 a_{j,m_j}$$
 for $j = 1,2$, we
 define
\begin{equation}\label{mul4}  g_1 g_2 =  c  a_{1,0} h_{1,1} a_{1,1} h_{1,2}
a_{1,2} \cdots (h_{1,m_1}  h_{2,1}) a_{2,1} h_{2,2} a_{2,2} \cdots
h_{2,m_2} \end{equation} when $a_{1,m_1}a_{2,0} = c \in C$,  or
\begin{equation}\label{mul5}  g_1 g_2 =  \a _1 a_{1,0} h_{1,1} a_{1,1} h_{1,2}
a_{1,2} \cdots h_{1,m_1} (a_{1,m_1}a_{2,0}) h_{2,1} a_{2,1}
h_{2,2} a_{2,2} \cdots h_{2,m_2} \end{equation} when
$a_{1,m_1}a_{2,0} \notin C$.

We write $A  \langle x;y;t\rangle$ for the free product $ A  *C\{
x;y;t\}$.

We have the natural embedding $C\{ x;y;t\} \to A  \langle
x;y;t\rangle$. For $g \in C\{ x;y;t\},$ we write $\bar g$ for its
natural image in $A  *C\{ x;y;t\}$.

\begin{defn} Suppose $\mathcal I  $ is a T-ideal of $C\{ x;y;t\}$,  for which $\mathcal I  \subseteq \id (A).$
 The
\textbf{relatively  free product}  $  A  \langle
x;y;t\rangle_{\mathcal I}$ of $A   $ and $C\{ x;y;t\}$
\textbf{modulo}~${\mathcal I}$ is defined as $(A   *_C C\{
x;y;t\})/\hat {\mathcal I}$, where $\hat {\mathcal I}$ is the
two-sided ideal ${\mathcal I (A   * _C C\{ x;y;t\})}$ consisting of
all evaluations on $A *C\{ x;y;t\}$ of polynomials from $\mathcal
I$.
\end{defn}

We can consider $A  \langle x;y;t\rangle_{\mathcal I} $ as the ring
of (noncommutative) polynomials but with coefficients from $A$
interspersed throughout, taken modulo the relations in $ \mathcal
I.$

This construction is universal in the following sense: Any
homomorphic image of $A  \langle x;y;t\rangle$ satisfying these
identities (from $\mathcal I$) is naturally a homomorphic image of
$A \langle x;y;t\rangle_{\mathcal I} $. Thus, we have:

\begin{lem}\label{embed1}
(i) For any $g_1, \dots, g_k, h_1, \dots, h_k $ in $ A  \langle
x;y;t\rangle$, there is a natural endomorphism $ A  \langle
x;y;t\rangle\to  A  \langle x;y;t\rangle$ which fixes $A$ and all
$t_i$ and sends $x_i \mapsto g_i,$ $y_i \mapsto h_i.$

(ii) For any $g_1, \dots, g_k, h_1, \dots, h_k $ in $  A  \langle
x;y;t\rangle_{\mathcal I}$, there is a natural endomorphism $$  A
\langle x;y;t\rangle_{\mathcal I}\to   A  \langle
x;y;t\rangle_{\mathcal I},$$ which fixes $A$ and all $t_i$ and
sends $x_i \mapsto g_i,$ $y_i \mapsto h_i.$
\end{lem}

Although difficult to describe explicitly, the relatively free
product is needed implicitly in all known proofs of the
Braun-Kemer-Razmyslov Theorem in the literature. {\it From now on,
we assume that ${\mathcal I}$ contains $\CAP_{n+1}$, so that we can
work with $\widehat{M}$}.

%
%
%
%
%
%
%
%
%
%
%
%

Let $\MM_A$ denote  the image of  $\MMO$ under substitutions to $A
$, i.e., the $C$-submodule of $\widehat {C\{ x,y ,t \}}$ consisting
of the images of all doubly alternating polynomials (in $x_1
\ldots,x_n,$ and in $y_1,\ldots,y_n$). In view of
Lemma~\ref{len.01}, the natural action of $\Obst_n(A)$ on $\MM_A$
respects multiplication by the $\delta^{(n)}_{k_,h}$-operators.


\begin{prop} $\Obst_n(A)\MM_A =0$.
\end{prop}
\begin{proof} If $a  \in \Obst_n(A)$, then $a\MMO  \in \mathcal I$, in view of Lemmas~\ref{extr1} and
\ref{module.10} and Proposition~\ref{sof.1}, so is 0 modulo
$\mathcal I$.
\end{proof}

\begin{cor} If $b\in A$ belongs to the $T$-ideal generated by
doubly alternating polynomials, then $  \Obst_n(A)b =0$.
\end{cor}

\begin{proof} The element $b$ belongs to the linear combinations images of $\MM_A$
under specializations $x_i\mapsto a_i$.
\end{proof}

By Step 7 of Section~\ref{structure}, this will complete the proof
of the nilpotence of $\operatorname{Jac}(A)$ when $C$ is a field, or
more generally of any nil ideal when $C$ is Noetherian, once  we
complete the proof of Proposition~\ref{len.2}.

\subsection{A more formal approach to Zubrilin's argument}

Rather than push immediately into $A$,  one can perform   these
computations first at the level of polynomials and then specialize.
This requires a bit more machinery, since it requires adjoining the
commuting indeterminates $\xi_{n,A}$ to the free product, but might
be clearer conceptually.

Note that $\widehat {C\{  t  \}}[\xi_{n, {C\{t  \}}}]=R \otimes _C
\widehat {C\{  t  \}}.$

\begin{lem}\label{module.11} 
 $\MM$ becomes an $\widehat {C\{t  \}}[\xi_{n, C\{ t
\}}]$-module  via the action given as follows:

 Order the $\xi_{k,h}$ as $\xi_ j = \xi_{k_j,h_j}$
for $1 \le j < \infty.$

For a letter $\xi_j = \xi_{k_j,h_j}$, define $$\xi_j \hat f = \hat
\delta^{(n)}_{k_j,h_j}(\hat f),
$$
and, inductively, $$\xi_j ^d\hat f = \hat
\delta^{(n)}_{k_j,h_j}(\xi_j ^{d-1}\hat f).
$$

For a monomial $h = \xi_j^{d_j}\dots \xi_1^{d_1}$ of degree $d =
d_1 + \cdots + d_j,$   define
$$h f = \xi_j^{d_j}(\xi_{j-1}^{d_{j-1}}\dots \xi_1^{d_1} \hat f)
$$
inductively on $j$.

Finally, define
$$\sum (c _i h_i)\hat f = \sum \a _i (h_i \hat f) $$
where $c_i \in C$ and $h_i$ are distinct monomials.

\end{lem}
\begin{proof}

 The action  is clearly well-defined, so we need to verify the
associativity and commutativity of the action. It is enough to show
that $(h_i h_{i'}) \hat f = h_i (h_{i'})\hat f$ for any two
monomials $h_i$ and $h_i'$. But this follows inductively from
induction on their length, plus the fact that $\xi_j (\xi_{j'}\hat
f) = \xi_{j'} (\xi_{j}\hat f)$ for any $\xi_j$ and $ \xi_{j'}.$
\end{proof}

%

Let us continue to take $\mathcal I = \mathcal{CAP}_{n+1}.$

\begin{rem} Clearly $A[\xi_{n,A}]  *_C C\{
 t \}\subset A[\xi_{n,A}]  *_C C\{ x;y;t\}$ in the natural way,
 and then
 $$\mathcal I (A[\xi_{n,A}]  *_C C\{  t\}) = (A[\xi_{n,A}]  *_C C\{t\})
 \cap \mathcal I (A[\xi_{n,A}]  *_C C\{
 x;y;t\})$$
since we are just restricting the indeterminates $\vec x,\vec
y,\vec t$ to the indeterminates $\vec t$.
\end{rem}

It follows from Noether's Isomorphism Theorem that $$ \mathcal F :
=  (A[\xi_{n,A}]  *_C C\{ t\})/{\mathcal I (A[\xi_{n,A}]  * _C
C\{t\})},$$ can be viewed naturally in $(A[\xi_{n,A}]  *_C C\{
x;y;t\})/{\mathcal I (A[\xi_{n,A}]  *_C C\{ x;y;t\})}$.

  Viewing $\MMO \subset C\{ x;y;t\}\subset A[\xi_{n,A}]  *_C C\{
x;y;t\}$, we define
\begin{equation} \tMM' = (A[\xi_{n,A}] *_C C\{ t\}) \MMO \subset A[\xi_{n,A}]  *_C C\{
x;y;t\},\end{equation} and its image    in $(A[\xi_{n,A}]
* _C C\{ x;y;t\})/ \mathcal I (A[\xi_{n,A}]  * _C C\{
 x;y;t\}),$ which we call $ \tMM $ (intuitively consisting of terms ending with images of doubly alternating polynomials), which
 acts naturally by
right multiplication on $\mathcal F $. To understand how $ \tMM $
works, we look at the Capelli polynomial acting on $A  *_C C\{
x;y;t\}$ for an arbitrary algebra $A$ satisfying $\Capl _{n+1}$.

%
%

There is a more subtle action that we need. $\tMM$ can be viewed
as an $R$-module where $R = C[\xi_{n,\widehat{C\{ x;y;t\}}}] $,
via the crucial Lemma~\ref{module.10}. But as above, $ \MMO$ is an
$A  * C \{ t\}$-module where the algebra multiplication is induced
from \eqref{mul} (viewing $ \MMO \subset C\{ x;y;t\}$), implying $
\MM$ is an $A
*  C \{ t\}$-module annihilated by
$CAP_{n+1}, $
%
%
%
 and $\tMM $ thereby becomes an $A[\xi_{n,\widehat{C\{ x;y;t\}}}]  * \widehat{C \{
t\}}$-module, where we define $$\xi_{k,h} \hat f =
\hat\delta^{(n)}_{k,h} \hat f$$ for $h \in \widehat{C\{x;y;t\}}$ and
$\hat f \in \MM $, by means of the action given
in~Lemma~\ref{module.11}, also cf.~Remark~\ref{obst1}. Our main task
is to identify these two actions when they are specialized to~$A$.

\subsubsection{The specialization argument}\label{S2.6} $ $

Having in hand the module $\tMM $ on which $A[\xi_{n,\widehat{C\{
x;y;t\}}}]$ acts,   we can specialize the assertion of
Proposition~\ref{sof.1} down to $A$ once we succeed in matching the
actions of $A[\xi_{n,\widehat{C\{ x;y;t\}}}]$ and $A[\xi_{n,A}] $
when specializing to $A$.

\begin{rem} $\mathcal{CAP}_k (A[\xi_{n,A}]) =  \mathcal{CAP}_k (A)[\xi_{n,A}]$, since $\Capl_k$ is
multilinear.
\end{rem}

We write ${\mathcal{DCAP}_n}$ for the $ C\{ t\} $-submodule of $ C\{
x;y;t\} $ generated by $\Dcap_n,$ cf.~\eqref{put1}, and
$\widehat{\mathcal{DCAP}_n}$ for its image in $ \widehat{C\{
x;y;t\}} $. This is a set of doubly alternating polynomials in $x_1,
\dots, x_n$ and $y_1, \dots, y_n$, with variables $t_i$ interspersed
arbitrarily.

\begin{lem}\label{spec1}
Any  specialization $\varphi: C\{ x;y;t\}  \to A$ (together with
its accompanying specialization $ \widehat{\varphi}: \widehat{C\{
x;y;t\}}  \to A$) gives rise naturally  to a map
$$\Phi : A[\xi_{n,A}]\widehat{\mathcal{DCAP}_n} \to  \tMM$$ given by
$$\sum _i a_i \xi_{k,\varphi (h_{j_i})} \widehat{ f_i} \mapsto \sum
_i a_i \widehat{\varphi}( \widehat{\xi_{k,h_{j_i}}}  f_i)=  \sum _i
a_i \widehat{\varphi}( \widehat{\delta_{k,h_{j_i}}^{(x,n)} f_i})$$
where $ \widehat{ f_i} \in \widehat{\mathcal{DCAP}_n}.$
\end{lem}
\begin{proof} We need to show that this is well-defined, which follows from the functoriality
property given in Lemma~\ref{extr1}. Namely, if $\varphi(h_{j_i})=
\varphi(h_{j_i}') ,$ then $ \widehat{\varphi}( \widehat{h_{j_i}})=
 \widehat{\varphi} \widehat{(h_{j_i}'}) $ and $$ \sum _i a_i
\varphi(\widehat{\delta_{k,h_{j_i}}^{(x,n)}} \widehat{ f_i}) =  \sum
_i a_i \varphi(\delta_{k,h_{j_i}}^{(x,n)}  f_i)= \sum _i a_i
 \varphi(\delta_{k,h_{j_i}'}^{(x,n)}  f_i)= \sum _i a_i
\varphi(\widehat{\delta_{k,h_{j_i}'}^{(x,n)}  f_i}).$$
\end{proof}

The objective of this lemma was to enable us to replace $A[\xi_{n,A}]$ by $A$ in our considerations.
%
 $\ker \Phi$ contains all  $\widehat{\delta_{k,h}^{( n)}
f} - \xi_{k,h} \hat f$  (cf.~Remark~\ref{welld})  as well as $ (\hat
h^n -\sum _{k=0}^{n-1} \hat h^k \xi_{k,h})\widehat{f}$, where $h$
ranges over all words and $\hat f\in \widehat{\mathcal{DCAP}_n}$, so
we see that the Zubrilin integrality relations are passed on.

\begin{lem}\label{loc} If  $t$ is an infinite set of
noncommuting indeterminates whose cardinality $\aleph$ is at least
that of~$A$, then for any given evaluation $w$ in
$\mathcal{DCAP}_{n}(A *_C C\{ x;y;t\}),$ there is a map
$$\varphi_w: C\{ x;y;t\} \to A *_C C\{ x;y;t\},$$ sending
$\mathcal{DCAP}_{n}$ to $\mathcal{DCAP}_{n}(A *_C C\{ x;y;t\}),$
such that   $w$ is in the image of $\varphi_w$.
\end{lem}

\begin{proof}  Note that $A  *_C C\{ x;y;t\}$ has cardinality  $\aleph$. Setting aside indeterminates $$\{t_g : g \in A  *_C C\{ x;y;t\}
\},$$ we still have $\aleph$ indeterminates left over, to map onto
our original set $t$ of  $\aleph$ indeterminates. But any evaluation
$w$ of $\mathcal{DCAP}_{n}$ on $A *_C C\{ x;y;t\} $ can be written
as \begin{equation}\label{spa} w=  g \Capl _n(x_1,\ldots, x_{n};
g_1, \dots, g_{n})g'   \Capl_n(y_1,\ldots ,y_n; h_1, \dots,
h_{n})g'' ,\end{equation} for suitable $g,g',g'',g_i,h_j \in A
*_C C\{ x;y;t\} $. Defining $\varphi_w$ by sending $x_i \mapsto
x_i $, $y_j \mapsto y_j$, and sending  the appropriate $t_{g}
\mapsto g,$ $t_{g'} \mapsto g',$  $t_{g''} \mapsto g'',$ $t_{g_i}
\mapsto g_i,$ and $t_{h_j} \mapsto h_j,$  we have an element in
$\varphi_w^{-1}(w)$.
\end{proof}

Clearly $\varphi_w(\mathcal{CAP}_{n+1}) \subseteq
\mathcal{CAP}_{n+1}(A *_C C\{ x;y;t\}),$ so,  when $\Capl_{n+1}\in
\mathcal I,$ $\varphi_w$ induces a map $$\hat \varphi_w:
\widehat{C\{ x;y;t\}} \to (A *_C C\{ x;y;t\})_{\mathcal I },$$which
sends $\MM \to \tMM.$

Although we do not see that $\mathcal{CAP}_{n+1}$ need be mapped
onto $\mathcal{CAP}_{n+1}(A *_C C\{ x;y;t\}),$ Lemma~\ref{loc} says
that it is ``pointwise'' onto, according to any chosen point, and
this is enough for our purposes.

\begin{thm}\label{last.3}
$\Obst_n(A)\cdot (\mathcal{CAP}_n(A))^2=0$, for any  PI-algebra $A
= C\{ a_1, \dots, a_\ell \}$ satisfying the Capelli identity
$\Capl _{n+1}$.
\end{thm}
\begin{proof}

We form the free algebra ${C\{ x;y;t\}}$  by taking a separate
indeterminate $t_j$ for each element of~$A[\xi_{n,A}]
\widehat{\mathcal{DCAP}_n}$. We work with $A[\xi_{n,A}]
\widehat{\mathcal{DCAP}_n},$ viewed in the relatively  free product
 $\tilde A : = (A[\xi_{n,A}]  *_C C\{ x;y;t\})_{\mathcal
I}$, where $\mathcal I = \mathcal{CAP}_{n+1}(A[\xi_{n,A}]  *_C C\{
x;y;t\}).$  In view of Lemma~\ref{spec1}, the relation
$$I_{n,\widehat {C\{ x;y;t\}}}\cdot \MM\equiv 0~ \pmod
{\mathcal{CAP}_{n+1}(A[\xi_{n,A}]  *_C C\{ x;y;t\})}$$ ~restricts
to the relation $I_{n,\widehat {C\{ x;y;t\}}
}\widehat{\mathcal{DCAP}_n} \equiv 0 \pmod
{\mathcal{CAP}_{n+1}(A[\xi_{n,A}]  *_C C\{ x;y;t\})}.$ 
But the various specializations of  Lemma~\ref{loc} cover all of
$\mathcal{DCAP}_n(A)$.
 Hence  Lemma~\ref{extr1} applied to Proposition~\ref{sof.1}
 and Lemma~\ref{spec1} implies $I_{n,  A
} \mathcal{DCAP}_n(A) = 0$, and thus
$$\Obst_n(A)\cdot (\mathcal{CAP}_n(A))^2 \subseteq I_{n,  A
} \mathcal{DCAP}_n(A) = 0.$$
\end{proof}


\subsection{The proof of Proposition~\ref{len.2}}\label{ster} $ $

\medskip
Now  we present the
 proof of the crucial Proposition~\ref{len.2},
stating that for a doubly
alternating polynomial $f=f(x_1,\dots,x_n,y_1,\dots,y_n,\vec{t})$,
$$\delta_{k,h}^{(x,n)}(f)\equiv\delta_{k,h}^{(y,n)}(f)\quad  modulo ~\mathcal{CAP} _{n+1}.$$
\medskip

\subsubsection{The connection to the group algebra of $S_n$}\label{ster1} $ $

We begin with the basic correspondence between multilinear
identities and elements of the group algebra over $S_n$.

$V_n=V_n(x_1,\ldots , x_n)$ denotes the $C$-module of multilinear
polynomials in $x_1,\ldots , x_n$, i.e.,
$$
V_n=\mbox{span}_C\{x_{\sg(1)} x_{\sg(2)}\cdots x_{\sg(n)} \mid \sg
\in S_n\}.
$$

\begin{defn}\label{identify}
We identify $V_n$ with the group algebra $C[S_n]$, by identifying
a permutation $\sg\in S_n$ with its corresponding monomial (in
$x_1,x_2,\ldots ,x_n)$:
$$
\sg\leftrightarrow M_\sg(x_1,\ldots ,x_n)=x_{\sg(1)}\cdots
x_{\sg(n)}.
$$

Any polynomial $\sum \alpha _\sg x_{\sg(1)} \cdots x_{\sg(n)}$
corresponds to an element $\sum \alpha _\sg \sg \in C[S_n],$ and
conversely, $\sum \alpha _\sg \sg$ corresponds to the polynomial
$$\left(\sum \alpha _\sg \sg\right) x_1 \cdots x_n = \sum \alpha _\sg x_{\sg(1)} \cdots x_{\sg(n)}.$$

\end{defn}

Here is a combinatorial identity of interest of its own.

\medskip
Consider two disjoint sets $X\cap Y=\emptyset$, each of
cardinality $n$, and the symmetric group $S_{2n}=S_{X\cup Y}$
acting on $X\cup Y$. For each subset $Z\subseteq X$ we define an
element $P(Z)\in C[S_{2n}]$ as follows:
$$
P(Z)=\sum_{\sigma(Z)\subseteq Y} \sgn(\sigma)\cdot\sigma.
$$
In particular
$$
P(\emptyset)=\sum_{\sigma\in S_{2n}} \sgn(\sigma)\cdot\sigma.
$$
\begin{prop}\label{jan.3}\label{PrZubrilinS2m}
\begin{eqnarray}\label{jan.3.2}
\sum_{Z\subseteq X}(-1)^{|Z|}P(Z)=\sum_{\sg(X)=X}\sgn(\sg)\cdot\sg.
\end{eqnarray}
\end{prop}


\begin{proof}
Let $\sg\in S_{2n}$ and let $a_\sg $ (resp.~$b_\sg$) be the
coefficient of $\sg$ on the l.h.s. (resp.~r.h.s.)
of~\eqref{jan.3.2}. We show that $a_\sg=b_\sg$. 

Let $Z(\sg)= \sigma^{-1}(Y)$ be the largest subset $Z\subseteq X$
such that $\sg(Z)\subseteq Y$. Note that $\sg(X)=X$
if and only if $Z(\sg)=\emptyset$. Therefore
\begin{eqnarray}\label{compare.1}
b_\sg=\sgn(\sg)\ \mbox{if}\ Z(\sg)=\emptyset\quad\mbox{and}\quad
b_\sg=0\ \mbox{if}\  Z(\sg)\ne\emptyset,
\end{eqnarray}
since $P(\emptyset)=\sum\sgn(\sg)\cdot\sg$. We claim that
$$
a_\sg=\sgn(\sg)\cdot \sum_{Z\subset Z(\sg)}(-1)^{|Z|}.
$$
To show this, recall that
$$
l.h.s=\sum_{Z\subseteq X}(-1)^{|Z|}\sum_{\sigma(Z)\subseteq Y} \sgn(\sigma)\cdot\sigma.
$$
In $P(Z)$ the coefficient of $\sg$ is $\sgn(\sg)$ if $Z\subseteq Z(\sg)$
 (since then $\sigma (Z)\subseteq Y$), and is zero if
$Z\not\subseteq Z(\sg)$ (since if $\sg (Z)\subseteq Y$ then
$\sg(Z\cup Z(\sg))\subseteq Y$, contradicting   the maximality of
$Z(\sg)$). It follows that as claimed,
$$
a_\sg=\sgn(\sg)\cdot\sum_{Z\subseteq Z(\sg)}(-1)^{|Z|}.
$$
It is well known that $\sum_{Z\subseteq Z(\sg)}(-1)^{|Z|}=1$ when $Z(\sg)=\emptyset$ and $=0$ otherwise.
Therefore
\begin{eqnarray}\label{compare.2}
a_\sg=\sgn(\sg)\quad\mbox{if}\quad Z(\sg)=\emptyset\quad\mbox{and}\quad a_\sg=0\quad\mbox{if}\quad Z(\sg)\ne \emptyset.
\end{eqnarray}
The proof now follows by comparing~\eqref{compare.1} with~\eqref{compare.2}.

\end{proof}


\begin{lem}\label{LeZubr2Grp}

Let $f(x_1,\dots,x_n,y_1,\dots,y_n,\vec{t}\;)$ be doubly
alternating. Then
$$f(x_1,\dots,x_n,y_1,\dots,y_n,\vec{t}\;)\equiv f(y_1,\dots,y_n,x_1,\dots,x_n,\vec{t}\;)\quad\mbox{
modulo} ~~\mathcal{CAP}_{n+1}.
$$
\end{lem}

\medskip
\begin{proof} Let $X = \{ x_1,\dots,x_n \}$ and $Y= \{
y_1,\dots,y_n\}.$ Then  $|X|=|Y|=n$ and $X\cap Y=\emptyset$, and we
identify $S_{2n}=S_{X\cup Y}$.  Let $M=\{x_{i_1},\ldots
,x_{i_k}\}\subseteq X$, with $1\le i_1<\cdots <i_k\le n$, and
$N=\{y_{j_1},\ldots ,y_{j_k}\}\subseteq Y$, with $1\le j_1<\cdots
<j_k\le n$. Thus, $|M|=|N|=k\le n$.  $M$ will play the role of $Z$
in Proposition~\ref{PrZubrilinS2m}. We consider permutations $\sg\in
S_{2n}$ with $\sg(M)=N$. Define the permutation
$$
\tau_{MN}=(x_{i_1},y_{j_1})\cdots (x_{i_k},y_{j_k}).
$$
Since $M\cap N=\emptyset$, $\tau_{MN}$ has order 2 in $S_{2n}$,
and
 satisfies $\sgn(\tau_{MN})= (-1)^k$.
 If  $M=X$ then $N=Y$ and
 $
 \sgn(\tau_{MN})= \sgn(\tau_{XY})=(-1)^n.
 $
 Moreover $\tau_{MN}(M)=N$ and $\tau_{MN}(N)=M$.

 \medskip

Next, we define
 $$
 T_{MN}=\sum_{\pi(M)=N} \sgn(\pi) \cdot\pi \in C[S_n].
 $$
   Let $\rho=\tau_{MN}\cdot\pi$, so that $\rho(M)=M$. Then      $\pi=\tau_{MN}\cdot\rho$
 and $$
 T_{MN}=\sgn({\tau_{MN}})\cdot \tau_{MN}\cdot\left(\sum_{\rho(M)=M}\sgn(\rho) \cdot\rho \right).
$$
But by Proposition~\ref{jan.3},
 \begin{eqnarray}\label{jan.5}
 \sum_{M\subseteq X}(-1)^{|M|}P(M)=\sum_{\sg(X)=X}\sgn\sg\cdot\sg.
 \end{eqnarray}
 If $M\subsetneq X$,  then $P(M)$ is alternating on $2n-|M|\ge n+1$ indeterminates, and hence is~$0$
 modulo $\mathcal{CAP} _{n+1}$. Thus, modulo
 $\mathcal{CAP} _{n+1}$, the left hand side of~\eqref{jan.5} equals the unique summand with $M=X$,
 which is
$$(-1)^n\sum_{\sg(X)=Y}\sgn(\sg)\cdot\sg=(-1)^n { \sgn({\tau_{XY}})}\cdot\tau_{XY}\cdot\left(\sum_{\sg(Y)=Y}\sgn(\sg)\cdot\sg\right)=
 \tau_{XY}\cdot\left(\sum_{\sg(Y)=Y}\sgn(\sg)\cdot\sg\right).
$$
Since $\sg(X)=X$ if and only if $\sg(Y)=Y$,
it follows that
$$
\sum_{\sg(Y)=Y}\sgn(\sg)\cdot\sg = \sum_{\sg(X)=X}\sgn(\sg)\cdot\sg
\equiv    \tau_{XY}\cdot
\left(\sum_{\sg(Y)=Y}\sgn(\sg)\cdot\sg\right),\quad \mbox{modulo}
~\mathcal{CAP} _{n+1}.
$$
Now we identify elements in $C[S_{2n}]$ with  polynomials
multilinear in $x_1,\ldots,x_n,y_1,\ldots,y_n$.
 Taking a monomial $h(x_1,\ldots,x_n,y_1,\ldots,y_n; \vec t)$  multilinear in
$x_1,\ldots,x_n,y_1,\ldots,y_n$, we define
$$
f(x_1,\ldots,x_n,y_1,\ldots,y_n;  \vec t) := \left(\sum_{\sg(Y)=Y}\sgn(\sg)\cdot\sg \right) h.
$$
Then
$$
\tau_{XY}\cdot
\left(\sum_{\sg(Y)=Y}\sgn(\sg)\cdot\sg\right)h=f(y_1,\ldots,y_n,x_1,\ldots,x_n;  \vec t).
$$
Again, since $\sg(X)=X$ if and only if $\sg(Y)=Y$, it follows that
$f(x_1,\ldots,x_n,y_1,\ldots,y_n;  \vec t)$ is doubly alternating, and we have
proved that
  $$
 f(x_1,\ldots,x_n,y_1,\ldots,y_n;  \vec t)\equiv  f(y_1,\ldots,y_n,x_1,\ldots,x_n ;\vec t)\quad\mbox {
modulo} ~~\mathcal{CAP}_{n+1}, $$ as desired.

\medskip

\end{proof}


\subsubsection{ Proof of Proposition~\ref{len.2}}\label{len.22}$
$

We may assume that $h$ is a new indeterminate $z$. Recall that
$$
\delta_{k,z}^{(x,n)}(f(x_1,\dots,x_n,y_1,\dots,y_n,\vec{t}\;  ))=~~~~~~~~~~~~~~~~~~~~~~~~~~~~~~~~~~~~~~~~~~~~~~~~~
$$
$$~~~~~~~~~~~~~~~~~~~~~~~~=\sum_{1\le i_1<\dots<i_k\le n}f(x_1,\dots,
x_n,y_1,\dots,y_n,\vec{t}\;  )|_ {x_{i_ u}\mapsto {zx_{i_ u}}};
\quad u=1,\dots, k,
$$
and
$$\delta_{k,z}^{(y,n)}(f(x_1,\dots,x_n,y_1,\dots,y_n,\vec{t}\;  ))=~~~~~~~~~~~~~~~~~~~~~~~~~~~~~~~~~~~~~~~~~~~~~~~~~
$$
$$~~~~~~~~~~~~~~~~~~~~~~~~=\sum_{1\le i_1<\dots<i_k\le n}f(x_1,\dots,
x_n,y_1,\dots,y_n,\vec{t}\; )|_{{y_{i_ u}}\mapsto{zy_{i_ u}}} ;
\quad u=1,\dots, k.
$$

Let $z'=1+\varepsilon z$, $\varepsilon$ being a central
indeterminant. Then clearly
\begin{eqnarray}\label{epsilon.01}
f(z'x_1,\dots,z'x_n,y_1,\dots,y_n,\vec{t}\;)=\sum_{k=0}^n
\varepsilon^k\cdot \delta_{k,z}^{(x,n)}(f(x_1,\dots,x_n,y_1,\dots,y_n,\vec{t}\;  ))
\end{eqnarray}
and
\begin{eqnarray}\label{epsilon.1}
f(x_1,\dots,x_n,z'y_1,\dots,z'y_n,\vec{t}\;)=\sum_{k=0}^n
\varepsilon^k \cdot \delta_{k,z}^{(y,n)}(f(x_1,\dots,x_n,y_1,\dots,y_n,\vec{t}\;  )).
\end{eqnarray}



By Equations~\eqref{epsilon.01} and~\eqref{epsilon.1}  it is enough to show that
$$f(z'x_1,\dots,z'x_n,y_1,\dots,y_n,\vec{t}\;  )\equiv f(x_1,\dots,x_n,z'y_1,\dots,z'y_n,\vec{t}\;  )
\quad\mbox { modulo} ~~\mathcal{CAP}_{n+1}.$$

Let
$$g_1(x_1,\dots,x_n,y_1,\dots,y_n,\vec{t}\;  )=f(z'x_1,\dots,z'x_n,y_1,\dots,y_n,\vec{t}\;  )$$
and
$$g_2(x_1,\dots,x_n,y_1,\dots,y_n,\vec{t}\;  )=f(x_1,\dots,x_n,z'y_1,\dots,z'y_n,\vec{t}\;  ).$$
We have to show that $$g_1\equiv g_2 \quad\mbox { modulo}
~~\mathcal{CAP}_{n+1}.$$

\medskip
Denote $x_i'=z'x_i, ~y_i'=z'y_i;~i=1,\dots,n.$ Then
$$
 g_1(x_1,\dots,x_n,y_1,\dots,y_n,\vec{t}\;  )= f(x_1',\dots,x_n',y_1,\dots,y_n,\vec{t}\;  )\equiv~~~~~~~~~~~~~~~~~~~~~~~~~~~~~~~~~~~~~~
$$
$$
\equiv f(y_1,\dots,y_n,x_1',\dots,x_n',\vec{t}\;  )=
$$
$$~~~~~~~~~~~~~~~~~
=g_2(y_1,\dots,y_n,x_1,\dots,x_n,\vec{t}\;  )\equiv~~~~~~~~~~~~~~~
$$
$$~~~~~~~~~~~~~~~~~~~~~~~~~~~~~~
\equiv g_2(x_1,\dots,x_n,y_1,\dots,y_n,\vec{t}\;  )
\quad\mbox { modulo} ~~\mathcal{CAP}_{n+1}.
$$
The congruences follow from Lemma~\ref{LeZubr2Grp}
since both $f$ and $g_2$ are doubly alternating.



\section{Proof of Kemer's ``Capelli Theorem,''}\label{section.strong}

To complete the proof of Theorem~\ref{BKR8}, it remains to present
an exposition of Kemer's ``Capelli Theorem,'' that any affine PI
algebra over a field $ F $ satisfies a Capelli identity $\Capl_n $
for large enough $n$. This is done by abstracting a key property
of  $\Capl_n $, called \textbf{spareseness}.

\begin{defn}\label{sparseB} A multilinear polynomial $g = \sum
\a_\sg x_{\sg(1)}\dots  x_{\sg(d)}$ is a {\bf sparse identity}
of~$A$ if, for any monomial $f(x_1, \dots, x_d;\vec t)$   we have
$$\sum \a _\sg f(x_{\sg(1)},\dots,  x_{\sg(d)};\vec t)\in \id (A).$$
\end{defn}

 See \cite[\S2.5.2]{BR}  for more detail. The major example of a sparse identity is the
 Capelli identity.
 One   proves rather quickly that any sparse identity implies a
Capelli identity, so it remains to show that any affine PI algebra
over a field satisfies a sparse identity. There are two possible
approaches, both using the classical representation theory of $S_n$.
One proof relies on ``the branching theorem,'' which requires
characteristic~0, and the other relies more on the structure of the
group algebra $F[S_n]$, also with the technique of ``pumping''
polynomial identities,  and works in arbitrary characteristic.

\subsection  {Affine algebras satisfying a sparse identity}
$ $

%
%

%


Sparse identities work well with the left lexicographic order $<$.
%
%
If $b_1<\cdots < b_m$ and  $1\ne \sg\in S_m$, then
$(b_1,\ldots,b_m)<(b_{\sg(1)},\ldots,b_{\sg(m)})$.
Any   sparse identity over a field yields a powerful {\bf sparse
reduction procedure}. Namely, we may assume $\a _{(1)} = 1; $
given $a_1, \dots, a_d$ in $A$, we can replace any term $f(a_1,
\dots , a_d)$ by
$$-\sum _{ \sg \ne 1} \a _\sg f(x_{\sg(1)}\dots  x_{\sg(d)}, x_{d+1}, \dots, x_n).$$
(The analogous assertion also holds for $c_d$.)

%

\begin{lem}\label{pump}
Let $A=C\{a_1,a_2\ldots \}$ be a PI algebra, satisfying a sparse
multilinear identity~$p=\sum_{\sg\in S_d}\beta_ \sg x_{\sg(1)}\cdots
x_{\sg(d)}$ of degree $d$, with   $ d\le n$, and let $M(x_1,\ldots ,
x_n;\vec y\ )$ be a monomial  multilinear in $x_1,\ldots , x_n$ and
perhaps involving extra indeterminates $ \vec y$. We consider
$\Delta=M(v_1,\ldots,v_n;\overline {\vec y} ),$ where
$v_1,\ldots,v_n$
are words in the generators  $a_1,a_2,\ldots$ and
$\overline {\vec y}$ is an arbitrary specialization of $\vec y$ in $A$.  Assume that $k$ of the $v_i$ satisfy $|v_i|\ge d$ (length as
words in $a_1,a_2,\ldots$). If $\ell\ge d$, then $\Delta$ is a linear
combination of monomials $\Delta'=M(v_1',\ldots,v_n';\overline {\vec
y}\, )$ where at most $\ell-1$ of the words $~v_i'$ have length $\ge
d$.

%

\end{lem}
This clearly implies that $A$ is spanned by monomials
$\Delta'=M(v_1',v_2',\ldots)$, with at most $d-1$ of the $~v_i'$ having
length $\ge d$.

\begin{proof}

\medskip {\bf Claim:} If  $|v_{i_1}|,\ldots,|v_{i_d}|\ge d$, then  $\Delta=M(v_1,\ldots,v_n;\overline {\vec y}\, )$ is a linear combination of terms
$\Delta'=M(v_1',\ldots,v_n';\overline {\vec y}\, )$ satisfying
$$
(|v_1'|,\ldots, |v_n'|) < (|v_1|,\ldots, |v_n|).
$$


\medskip
The above Claim implies the existence of descending sequences of
monomials, under the left lexicographic order. Such a descending
sequence must stop. When it stops we have a corresponding monomial
having strictly fewer words $v_i'$ for which $|v_i'|\ge d.$ Therefore
proving the above Claim will  prove the lemma. We now prove the
Claim.

\medskip

We rewrite $\Delta=M(v_{i_1},\ldots,v_{i_d};\overline
{\vec y}\, )$, where $i_1 < i_2, \dots < i_d;$ then we may assume that $i_1 = 1, \dots, i_d = d.$
We  write $v_i=w_iu_i$ where $|u_i|=d-i$, $1\le i\le
d$. The sparse identity $p$ implies that $\Delta$
 is a linear combination of terms
 $\Delta _\sg=M(w_1u_{\sg(1)},\ldots,w_hu_{\sg(d)};\overline {\vec y}\, )=M(v_1',\ldots,v_d';\overline {\vec y}\, )$
 where
 $1\ne \sg\in S_d$.
 ($\Delta$ itself corresponds to $\sg=1$.)
 \medskip
 To see this, we rewrite $\Delta=M(w_1u_1,\ldots, w_du_d;\overline {\vec y}\, )$    as $N(u_1,\ldots, u_d;\overline{W}).$
  The sparse identity $p$
 implies that  $ N(u_1,\ldots, u_d;\overline{W})$  is a linear combination of elements of the form
 $$
 N(u_{\sg(1)},\ldots, u_{\sg(d)};\overline{W})=M(w_1u_{\sg(1)},\ldots,w_du_{\sg(d)};\overline {\vec y}\, ), ~~1\ne\sg\in S_d.
 $$
Denote $w_iu_{\sg(i)}=v'_i$, $1\le i\le d$. But then
$(|v_1'|,\ldots,|v_d'|)<(|v_1|,\ldots,|v_d|)$  for such $\sg\ne 1$. This proves the
 Claim, and completes the proof of the lemma.
\end{proof}

Although we did not apply Shirshov's Height Theorem,
 the main argument here is similar.
Note also that Lemma~\ref{pump} applies to any PI algebra, not
necessarily affine. In the next theorem, due to Kemer,
we do assume that $A$ is affine.

\begin{thm}\label{kemer.capelli}
Let $A=C\{a_1,\ldots,a_r\} $ be an affine   PI algebra over a
commutative ring $ C $, satisfying a sparse identity $p$ of degree
$d$, and let $n\ge r^d+d$. Then $A$ satisfies the Capelli identity
$\Capl_n[x;y]$.
\end{thm}
\begin{proof}
We may assume that $r\ge 2$, since otherwise $A$ is commutative.
Consider $$\Capl_n(v_1,\ldots, v_n; w_1,\ldots, w_n )$$ where
$v_i, w_i\in A$. By Lemma~\ref{pump} we may assume that at most
$d-1$ of the $v_i$ have length $\ge d$ (as words in the generators
$a_1\ldots,a_r$). Hence at least $n-(d-1)$ of the $v_i$ have length $\le
d-1$. The number of distinct words of length $q$ is $\le r^q$.
Hence the number of words of length $\le d-1$ is
$$
\le 1+r+r^2+\cdots +r^{d-1}=\frac{r^d-1}{r-1}<r^d\quad\mbox{(since
$r\ge 2$)}.
$$
But we have at least $n-(d-1)$ such words appearing in  $v_1,\ldots,v_n$,
and $n-(d-1)> r^d$ (since by assumption $n\ge r^d+d$). It follows
that there must be repetitions  among $v_1,\ldots,v_n$, so
$\Capl_n(v_1,\ldots,v_n;w_1,\ldots, w_n )=0$.
\end{proof}

\subsection{Actions of the group algebra}\label{preliminaries4kemer}
$ $

It remains to prove the existence of sparse identities for affine
PI-algebras. For this, we turn to the representation theory of
$S_n$. After a brief review of actions of $S_n$ on Young diagrams,
we treat the characteristic 0 case, cf.~Kemer~\cite{kemer.0.5}.
The characteristic $p>0$ proof, which requires some results about
modular representations but bypassing branching, is done in
\S\ref{preliminaries4kemer} and \S\ref{charp}.

 Given $\sg,\pi\in
S_n$, by convention we take $\sg \pi (i) = \pi(\sg(i)).$ The product
$\sg\pi$  corresponding  (by Definition ~\ref{identify})  to the
monomial
$$M_{\sg\pi} = x_{\sg\pi (1)}\cdots x_{\sg\pi (n)}$$ can be interpreted
in two ways, according to  {\bf left} and  {\bf right} actions of
$S_n$ \index{left action of $S_n$} \index{right action of $S_n$}
on $V_n$, described respectively as follows:

Let $\sg,\pi\in S_n$.  Let $y_i=x_{\sigma(i)}$. Then

\medskip

 (i) $\quad \sg M_\pi(x_1\ldots , x_n)  := M_{\sg\pi} =
M_\pi(x_{\sg (1)},\ldots ,x_{\sg (n)})\quad$ and

(ii) $\quad M_\sg (x_1\ldots , x_n)\pi := (y_1\cdots y_n)\pi=
M_{\sg\pi}  = y_{\pi (1)}\cdots y_{\pi (n)}. $

%

\noindent Thus, the effect of the right action  of $\pi$ on a
monomial is to permute the places of the indeterminates according
to $\pi$.

\medskip

Extending  by linearity, we obtain for any $f=f(x_1,\ldots ,
x_n)\in V_n$ that

\begin{enumerate}\item[(i)]
$\sg p(x_1,\ldots , x_n)= p(x_{\sg (1)} ,\ldots , x_{\sg (n)}); $
 \item[(ii)]\
$ p(x_1,\ldots , x_n)\pi=q(y_1,\ldots , y_n)$, where $q(y_1,\ldots
, y_n)$ is obtained from $p(x_1,\ldots , x_n)$ by place-permuting
all the monomials of $p$ according to the permutation $\pi$.
\end{enumerate}

For any finite group $G$ and field $F$, there is a well-known
correspondence between the $F[G]$-modules and the representations
of $G$. The simple modules correspond to the irreducible
representations.

\begin{rem} If $p\in \Id (A)$, then $\sigma p\in \Id (A)$ since the left action is just a change of
variables.

Hence, for any PI-algebra  $A$, the spaces
\[
\Id (A)\cap V_n\subseteq V_n
\]
are in fact left ideals of $F[S_n]$ (thereby affording certain
$S_n$ representations), but
need not be two-sided ideals. However, we prove below the existence
of a nonzero two-sided ideal in $\Id(A)\cap V_n$, a fact which is of
crucial importance in what follows.
%
\end{rem}

\begin{rem}\label{semiidem}

 Let $\lambda$ be a partition. As explained in \cite[p.~147]{BR}, any tableau
 $T$ of $\lambda$
gives rise to an element  $$a_T = \sum _{q\in {\mathcal
C_{T_\lambda}},\ p\in {\mathcal R_{T_\lambda}}} \sgn(q) qp \in
C[S_n],$$ where $ \mathcal C_{T_\lambda}$ (resp.~ $ \mathcal
R_{T_\lambda}$ ) denotes the set of column (resp.~row) permutations
of the tableau $T_\lambda$. $a_T^2 = \a _T a_t$ for some $\a_T$ in
the base field $F$. When $ \a _T \ne 0$, which by
\cite[Lemma~19.59(i)]{rowen5} is always the case when
$\operatorname{char}(F)$ does not divide $n$, in particular, when
$\operatorname{char}(F)=0$,
 we will call the idempotent
$e_T: =  \a _T^{-1} a_T $ the \textbf{Young symmetrizer} of the
tableau $T$.

Furthermore, by \cite[Lemma~19.59(i)]{rowen5}, if $a_T \ne 0$ and
then $F[S_n] a_T = Fa_T,$ implying $F[S_n] a_T $ (if nonzero) is a
minimal left ideal, which we call $J_\lm. $ Thus, if $J_\lm$
contains an element corresponding to a nontrivial PI of $A$, $a_T$
itself must correspond to a PI of $A$.

 $s^\lm : = \dim J_\lm $
is given by the ``hook" formula, see for example~\cite{sagan} or
\cite{jameskerber}, where we recall that each ``hook" number $h_x$
for a box $x$ is the number of boxes in ``hook" formed by taking all
boxes to the right of $x$ and beneath $x$. (In the literature, one
writes $f^\lm$ instead of $s^\lm$, but here we have used $f$
throughout for polynomials.)
\end{rem}

 \begin{lem}\label{hooked} Suppose $L$ is a minimal left ideal of
 a ring $R$. Then the minimal two-sided ideal of $R$ containing
 $L$ is a sum of minimal left ideals of $R$ isomorphic to $L$ as
 modules.
\end{lem}

We let $I_\lm$ denote the minimal two-sided ideal of $F[S_n]$
containing $J_\lm$.

We define the \textbf{codimension} $c_n(A)=\dim
\left(\frac{V_n}{\Id(A)\cap V_n}\right).$ The characteristic 0
version of the next result is in \cite{regev5}.

\begin{lem}\label{hook4}
Let  $A$ be an $F$-algebra, and let $\lm$ be a partition of $n$. If
$\dim J_\lm>c_n(A)$, then $I_\lm\subseteq \Id(A)\cap V_n$.
\end{lem}

\begin{proof}
 By Lemma~\ref{hooked},  $J_\lm$ is a sum of minimal left ideals, with
 each such  minimal left ideal~$J$ isomorphic to $J_\lm$.
Thus, $ \dim J = \dim J_\lm >c_n(A)$. Since $J$ is minimal, either
$J\subseteq \Id(A)\cap V_n$ or $J\cap\big (\Id(A)\cap V_n\big)=0$.
If $J\cap\big(\Id(A)\cap V_n\big)=0$ then it follows that \[c_n(A)=
\dim V_n/\big(\Id(A)\cap V_n\big)\ge \dim J>c_n(A), \] a
contradiction. Therefore each $J\subseteq \Id(A)\cap V_n$.
$I_\lm\subseteq \Id(A)\cap V_n$ since $I_\lm$ equals the sum of
these minimal left ideals.
\end{proof}

\subsection{The characteristic 0 case ~\cite{kemer.0.5}}$ $
The     characteristic 0 case is treated separately here, since it
can be handled via the classical representation theory of the
symmetric group.  By Maschke's Theorem, the group algebra $FS_n$
now is a finite direct product of matrix algebras over $F$. We
have the decomposition $F S_n=\bigoplus_{\lm\vdash n}I_\lm$.

Thus, Lemma~\ref{hook4} yields at once:

\begin{lem}\label{hook44}~\cite{regev5}
Let $\operatorname{char}(F)=0$, let $A$ be an $F$ algebra, and let
$\lm$ be a partition of $n$. If $s^\lm>c_n(A)$, then $I_\lm\subseteq
\Id(A)\cap V_n$.
\end{lem}

 (Here   $I_\lambda$ is the sum of those $F[S_n] e_T$ for
which $T$ is a standard tableau with partition $\lambda$.  These
$I_\lambda$ are minimal two sided ideals, each a sum of $s^\lm$
minimal left ideals isomorphic to $J_\lm$.)

\begin{example}\label{rect} Consider the ``rectangle'' of $u$ rows and $v$
columns. By ~\cite[page 11]{macdonald}, the hook numbers of the
partition $\mu=(u^v)$ satisfy
\[\sum_{x\in \mu }h_x=uv(u+v)/2=n\frac{ u+v}2. \]
Let us review the proof, for further reference. For any box $x$ in
the $(1,j)$ position, the hook has length $u+v-j$, so the sum of
all hook numbers in the first row is $$\sum _{j=1}^v (u+v-j) = uv
+ \frac {v(v-1)}2 = v\left(u + \frac{v-1}2\right).$$ Summing this
over all rows yields $$v\frac {u(u+1)}2 + uv\frac{v-1}2 = uv
\left(\frac{u+1}2+\frac{v-1}2\right)  = uv\frac{ u+v}2,$$ as
desired.
\end{example}

\subsubsection{Strong identities}$ $


\begin{defn}\label{strong.1}
Let $A$ be a PI algebra. The multilinear polynomial $g\in V_n$ is a
{\it strong} identity of $A$ if for every $m\ge n$ we have
$FS_m\cdot g\cdot FS_m \subseteq \Id(A)$.
\end{defn}

Note that every strong identity is sparse. To obtain strong
identities, we utilize the following construction, due to Amitsur.

The natural embedding $S_n\subset S_{n+1}$ (via $\sg(n+1)=n+1$ for
$\sg\in S_n$) induces the embedding $V_n\subset V_{n+1}$:
$f(x_1,\ldots ,x_n)\equiv f(x_1,\ldots ,x_n)\cdot x_{n+1}$. More
generally, for any $n<m$ we have the inclusion $V_n\subset V_m$
via $f(x_1,\ldots ,x_n)\equiv f(x_1,\ldots ,x_n)\cdot
x_{n+1}\cdots x_m$.

\medskip
For $f(x)=f(x_1,\ldots,x_n)=\sum_{\sg\in S_n}\alpha_\sg
x_{\sg(1)}\cdots x_{\sg(n)}\in V_n$,   we define
\begin{equation}\label{right.0.4}  f^*(x_1,\ldots,x_n;x_{n+1},\ldots,x_{2n-1})
=\sum_{\sg\in S_n}\alpha_\sg
x_{\sg(1)}x_{n+1}x_{\sg(2)}x_{n+2}\cdots
x_{\sg(n-1)}x_{2n-1}x_{\sg(n)}\end{equation}  $$\qquad
=(f(x_1,\ldots,x_n) x_{n+1}\cdots x_{2n-1})\eta,$$

%
where $\eta\in S_{2n-1}$ is the permutation
\begin{eqnarray}\label{september.7.2}
\eta=\left ( \begin{array}{ccccccc}
1~ &2~ &3~ &4~ &\cdots~ &2n-1  \\
1~ &n+1~ &2~ &n+2~ &\cdots~ &n
\end{array} \right ).
\end{eqnarray}
Let $L\subseteq\{x_{n+1},\ldots,x_{2n-1}\}$ and denote by $f^*_L$
the polynomial obtained from $f^*$ by substituting $x_j\rightarrow
1$ for all $x_j\in L$. Rename the indeterminates in
$\{x_{n+1},\ldots,x_{2n-1}\}\setminus L$ as
$\{x_{n+1},\ldots,x_{n+q}\}$ (where  $q=n-1-|L|$) and denote the
resulting polynomial as  ${ f}^*_L$. Then similarly
to~\eqref{right.0.4}, there exists a permutation $\rho\in S_{n+q}$
such that ${ f}^*_L=(f x_{n+1}\cdots x_{n+q})\rho$.

\medskip
Note that if $1\in A$ and $f^*\in \Id(A)$, then also ${ f}^*_L\in
\Id(A)$
for any such $L$, and in particular $f\in \Id(A)$. The converse is
not true: it is possible that $f\in \Id(A)$ but $f^*\not\in \Id(A)$.

\begin{lem}\label{right5}
Let $A$ be a PI~algebra, let $I\subseteq  V_n$ be a two--sided ideal
in $V_n$,  and assume  for any $f\in I$ that $f^*\in \Id (A)$ (and
thus $f\in \Id(A)$). Then for any $m\ge n$,
$$(FS_m)I(FS_m)\subseteq \Id(A).$$
\end{lem}
\begin{proof} Since~$(FS_m)I\subseteq
\Id(A)$, it suffices to prove:

\medskip
{\bf Claim:} If $f\in I$ and $\pi\in S_m$, then $f^*_L
\pi=(f(x_1,\ldots,x_n)x_{n+1}\cdots x_m)\pi\in \Id(A)$.

\medskip
 If
$f=\sum_{\sg\in S_n}a_\sg\sg (x_1\cdots x_n\cdots x_m)$, then
$f^*_L \pi=\sum_{\sg\in S_n}a_\sg\sg(\pi(x_1\cdots x_n\cdots
x_m))$.

\medskip
Consider the positions of $x_1,\ldots,x_n$ in the monomial
$\pi(x_1\cdots  x_m)$: There exists $\tau\in S_n$ such that
$$\pi(x_1\cdots x_n\cdots x_m)=g_0x_{\tau(1)}
g_1x_{\tau(2)}g_2\cdots g_{n-1}x_{\tau(n)}g_n=\tau(g_0x_1
g_1x_2g_2\cdots g_{n-1}x_ng_n) ,$$ where each $g_j$ is $=1$ or is
a monomial in some of the indeterminates $x_{n+1},\ldots, x_m$. It
follows that $f^*_L \pi=(f\tau)(g_0x_1 g_1x_2g_2\cdots
g_{n-1}x_ng_n)$. Since $f\in V_n$ and $\tau\in S_n$, $f\tau$ only
permutes the indeterminates $x_1,\ldots,x_n$, and hence
(see~\eqref{right.0.4})
\[f^*_L \pi=(f\tau)(g_0x_1 g_1x_2g_2\cdots
g_{n-1}x_ng_n)=g_0((f\tau)^*[x_1,\ldots,x_n;g_1,\ldots,g_{n-1}])g_n.\]
Since $I$ is two-sided, $f\tau\in I$, hence by assumption
$(f\tau)^*\in \Id (A)$, which by the last equality implies that
$f\pi\in \Id(A)$.
\end{proof}


\subsubsection{Existence of nonzero two-sided ideals $I_\lm\subseteq FS_n$ of identities}\label{late.need1}$ $


Let $c_n(A)\le \alpha^n$ for all $n$. The next lemma yields
rectangles $\mu=(u^v)\vdash n$ such that $\alpha^n<s^\mu$.

\begin{lem}\label{hook5}
Let $0<u,v$ be integers and let $\mu $ be the $u\times v$
rectangle $\mu =(u^v)\vdash u\cdot v$. Let $ n =uv$.
Then
\[\left(\frac{n}{u+v}
\right)^n\cdot\left(\frac{2}{e} \right)^n <s^\mu  \qquad
(\mbox{where} \quad e=2.718281828\ldots).\] In particular, if
$\alpha\le \frac{n}{u+v}\cdot\frac{2}{e}$ then $\al^n\le s^\mu$.
\end{lem}
\begin{proof}

Since the geometric mean is bounded by the arithmetic mean,
\[\left(\prod_{x\in\mu }h_x \right )^{1/n}\le\frac{1}{n}\sum_{x\in\mu }h_x
=\frac{u+v}{2},\] in view of Example~\ref{rect},
 and hence $$\left(\frac{2}{u+v}\right)^n\le
  \frac{1}{\prod_{x\in\mu }h_x }.$$

Together  with the classical inequality
$(n/e)^n<n!$, this implies that
\[
\left(\frac{uv}{u+v} \right)^n\cdot\left(\frac{2}{e}
\right)^n=\left(\frac{n}{e} \right)^n\cdot\left(\frac{2}{u+v}
\right)^n <\frac{n!}{\prod_{x\in\mu }h_x}=s^\mu .\]
\end{proof}

\begin{rem}\label{regest}To apply this, we need Regev's estimate ~\cite{regev1}
 of codimensions, $$c_m(A)\le (d-1)^{2m},$$ as explained in
\cite[Theorem~5.38]{BR}.
\end{rem}

\begin{prop}\label{hook6}~\cite{amitsur.regev}
Let $A$ be a PI~algebra satisfying an identity of degree $d$.
Choose natural numbers $u$ and $v$ such that
$$\frac{uv}{u+v}\cdot\frac{2}{e}\ge (d-1)^4\,.\qquad\mbox{For example, choose
$~~ u=v\ge e\cdot (d-1)^4$}.$$ Let $n=uv$ and let $\mu=(u^v)$ be the
$u\times v$ rectangle. Let $n\le m\le 2n$ and let $\lm\vdash m$ be
any partition of $m$ which contains $\mu$: $(u^v)\subseteq \lm$.
Then the elements of the corresponding two--sided ideal
$I_\lm\subseteq FS_m$ are identities of $A$: $~I_\lm\subseteq
\Id(A)\cap V_m $.
\end{prop}

\begin{proof}
Since $m\le 2n$, $(d-1)^{2m}\le (d-1)^{4n}$, and by assumption
$(d-1)^4\le\frac{n}{u+v}\cdot\frac{2}{e}$. By Lemma~\ref{hook5},
$\left(\frac{n}{u+v}\cdot\frac{2}{e}\right)^n<s^\mu$ and since
$\mu\subseteq \lm$, we know that $s^\mu\le s^\lm$. Thus, by
Remark~\ref{regest},

\[c_m(A)\le (d-1)^{2m}\le (d-1)^{4n}\le\left(\frac{uv}{u+v}\cdot\frac{2}{e}\right)^n<s^\mu\le
s^\lm,\] and the assertion  now follows from Lemma~\ref{hook44}.
\end{proof}

\begin{cor} Hypotheses as in Proposition~\ref{hook6}, for $n\le m\le 2n$,
\[\bigoplus_{\lm\vdash m\atop \mu\subseteq\lm}I_\lm\subseteq \Id(A).\]
Consequently, if $f\in I_\mu$ then $f^*\in \Id(A)\cap V_{2n-1}$
(see~\eqref{right.0.4}). Also, for any subset
$L\subseteq\{n+1,\ldots,2n-1\}$,  $f^*_L\in \Id(A)$, and in
particular $f\in \Id(A)$.
\end{cor}

\begin{proof}
By ``branching,'' the two--sided ideal generated in $V_m$ by
$I_\mu$ is
\[V_mI_\mu V_m=(FS_m)I_\mu (FS_m)=\bigoplus_{\lm\vdash m\atop \mu\subseteq\lm}I_\lm.\]
Hence, $(FS_m)I_\mu (FS_m)\subseteq \Id(A)$ for any $n\le m\le
2n-1$, and in particular, if $f\in I_\mu$ and $\rho\in S_m$ then
$f\rho\in \Id(A)$.  \eqref{right.0.4} concludes the proof.
\end{proof}

 By  Proposition~\ref{hook6} and Lemma~\ref{right5} we have just
proved

\begin{prop}\label{strong.2}
Every PI algebra in characteristic $0$  satisfies non-trivial
strong identities.
\\ Explicitly,
let $\operatorname{char}(F)=0$ and let $A$ satisfy an identity of
degree $d$. Let $u,v$ be
 natural numbers such that
$\frac{uv}{u+v}\cdot\frac{2}{e}\ge (d-1)^4$, and let  $\mu=(u^v)$
be the $u\times v$ rectangle. Then every $g\in I_\mu$ is a strong
identity of $A$. The degree of such a strong identity $g$ is $uv$.
We can choose for example
 $u=v=\lceil{ e\cdot(d-1)^4}\rceil$, so
 $\deg (g)=\lceil{ e\cdot(d-1)^4}\rceil^2= e^2(d-1)^8$.
\end{prop}

\medskip
We summarize:
\begin{thm}\label{Regcon}
 Every affine PI algebra over a field of characteristic 0 satisfies some Capelli identity.
 Explicitly, we have the following:

\medskip
(a) ~Suppose the $F$-algebra $A$ satisfies an identity of degree
$d$. Then $A$ satisfies a strong identity of degree
$$
d'=\lceil e(d-1)^4\rceil^2= e^2(d-1)^8.
$$

\medskip
(b) ~Suppose $A=F\{a_1,\ldots,a_r\}$, and  $A$ satisfies an
identity of degree $d$ and take $d'$ as in~(a). Let $n= r^{d'}+d'
\approx r^{e^2(d-1)^8}$. Then $A$ satisfies the Capelli identity
$\Capl_n$.
\end{thm}
\begin{proof}  (a) is by Proposition~\ref{strong.2}, and then (b) follows from
Theorem~\ref{kemer.capelli}, since every strong identity is
sparse.\end{proof}

\subsection{Actions of the
group algebra on sparse identities}\label{preliminaries4kemer}
$ $

Although the method of \S\ref{preliminaries4kemer} is the one
customarily used in the literature, it does rely on branching and
thus only is effective in characteristic 0. A slight modification
enables us to avoid branching. The main idea is that any sparse
identity follows from an identity of the form
 $$f = \sum _{\sg \in S_n}
\a_\sg x_{\sg(1)}x_{n+1}\cdots  x_{\sg(n)}x_{2n},$$ since we could
then specialize $x_{n+1}, \dots, x_{2n}$ to whatever we want. Thus,
letting $V'_n$ denote the subspace of $V_{2n}$ generated by
 the words $x_{\sg(1)}x_{n+1}\cdots  x_{\sg(n)}x_{2n},$ we
can  identify the sparse identities  with $F[S_n]$-subbimodules of
$V'_n$ inside $V_{2n}.$ But there is an  as $F[S_n]$-bimodule
isomorphism $ \varphi: V_n \to V'_n$, given by $x_{\sg(1)} \cdots
x_{\sg(n)} \to   x_{\sg(1)}x_{n+1}\cdots  x_{\sg(n)}x_{2n}.$ In
particular $V'_n$ has the same simple $F[S_n]$-subbimodules
structure as $V_n$ and can be studied with the same Young
representation theory, although now we only utilize the left action
of permutations.

Thus, for any PI-algebra  $A$, the spaces
\[
\Id (A)\cap V'_n\subseteq V'_n
\]
are   $F[S_{n}]$-subbimodules of $V'_n$.

\begin{rem}\label{semiidem11}
Again, any tableau $T$ of $2n$ boxes gives rise to an element $$a_T
= \varphi \left(\sum _{q\in {\mathcal C_{T_\lambda}},\ p\in
{\mathcal R_{T_\lambda}}} \sgn(q) qp\right) \in F[S_{2n}],$$ where $
\mathcal C_{T_\lambda}$ (resp.~ $ \mathcal R_{T_\lambda}$ ) denotes
the set of column (resp.~row) permutations of the tableau
$T_\lambda$.

Thus,  $F[S_n] a_T $ (if nonzero) is an $F[S_n]$-submodule, which we
call $J_\lm. $ If $J_\lm$ contains an element corresponding to a
nontrivial PI of $A$, $a_T$ itself must correspond to a PI of $A$.
\end{rem}

We let $I_\lm$ denote the minimal $F[S_n]$-bisubmodule of
$F[S_{2n}]$ containing $J_\lm$.

\begin{lem}\label{hook45}
Let  $A$ be an $F$-algebra, and let $\lm$ be a partition of $n$. If
$\dim J_\lm>c_{2n}(A)$ and $J_\lm$ is a simple $F[S_n]$-module, then
$I_\lm\subseteq \Id(A)\cap V'_n$.
\end{lem}

\begin{proof}
 Same as Lemma~\ref{hook4}, noting that $I_\lm$ is a sum of $F[S_n]$-submodules $J_\lm a$ each isomorphic to $J_\lm$.
Thus, taking such $J$, one has
\[c_n(A)=\dim \left(\frac{V'_n}{\Id(A)\cap V'_n}\right)\ge \dim J >c_{2n}(A),\] a
contradiction. Therefore each $J\subseteq \Id(A)\cap V'_n$, implying
$I_\lm\subseteq \Id(A)\cap V_n$.
\end{proof}

Note that when $\operatorname{char}(F)=p>0$, the lemma might fail
unless $J_\lm$ is simple. James and Mathas \cite[Main
Theorem]{jamesmathas} determined when $J_\lm$ is simple for $p=2$.

One such example is when $\lm$ is the \textbf{staircase}, which we
define to be the Young tableau $T_{u}$ whose
 $u$ rows have  length  $ u, u-1, \dots, 1. $ This gave rise to the James-Mathas conjecture
\cite{mathas} of conditions on $\lm$ characterizing when $J_\lm$ is
simple in characteristic $p>2$, which was solved by Fayers~\cite{F}.

\section{Kemer's Capelli Theorem for all characteristics}\label{charp}

In this section we give a  proof of Kemer's ``Capelli Theorem'' over
a field of any characteristic. In fact in characteristic $p$ Kemer
proved a stronger result, even for non-affine algebras.

\begin{thm}\label{capelli.kemerp}~\cite{kemerp}
 Any PI algebra over a field $ F $ of characteristic $p>0$  satisfies a Capelli identity
$\Capl_n$ for large enough $n$.
\end{thm}

This fails in characteristic 0, since  the Grassmann algebra
  does not satisfy a Capelli identity. The proof of Theorem~\ref{capelli.kemerp} given in \cite{kemerp}
  is
quite complicated; an elementary proof using the ``identity of
algebraicity'' is given in \cite[\S2.5.1]{BR}, but still requires
some computations. In the spirit of providing a full exposition
which is as direct as possible, we treat only the affine case via
representation theory, in which case characteristic $p>0$ works
analogously to characteristic $0$. This produces a much better
estimate of the degree of the sparse identity, which we  obtain in
Theorem~\ref{sparse}.

In view of Theorem~\ref{kemer.capelli}, it suffices to show that any
affine PI algebra satisfies a sparse identity. Although we cannot
achieve this through branching, the ideas of the previous section
still apply, using \cite{F}.

\subsection{Simple Specht modules in characteristic $p>0$} $ $

In order to obtain a $p$-version of Proposition~\ref{hook6} in
characteristic $p>2$, first we need to find a class of partitions
satisfying Fayer's criterion.

 For a
positive integer $m,$ define $v_p$ to be the $p$-adic valuation,
i.e., $v_p(m)$ is the largest power of $p$ dividing $m$. Also,
temporarily write $h_{(i; j)} $ for $h_x$ where $x$ is the box in
the $i,j$ position. The James-Mathas conjecture for $p \ne 2$,
proved in \cite{F}, is that $J_\lambda$ is simple  if and only if
there do not exist $i,j,i',j'$ for which $v_p(h_{(i; j)}) > 0$ with
$v_p(h_{(i; j)}), v_p(h_{(i'; j)}), v_p(h_{(i; j')})$ all distinct.
Of course this is automatic when each hook number is prime to $p$,
since then every $v_p(h_{(i; j)}) = 0$.
%

\begin{example}\label{rect1}  A \textbf{wide staircase} is a Young tableau   $T_{u}$ whose
 $u$ rows have all have lengths different
multiples of $p-1$, the first row of length $(p-1)u,$ the second of
length $(p-1)(u-1),$ and so forth until the last of length $p-1$.
The number  of boxes is $$n = \sum _{j=1}^u (p-1)j = (p-1)\binom
{u+1}2.$$

When $p=2,$ the wide staircase just becomes the staircase described
earlier.

 In analogy to Example \ref{rect}, the
dimension of the ``wide staircase'' $T_{u}$ can be estimated as
follows:   We write $j = (p-1)j' + j''$ for $1 \le j'' \le p-1.$ The
hook of a box   in the $(i,j)$ position has length
$(u+1-i)(p-1)+1-j$, and depth $u+1-j'-i,$ so the hook number is
$$(u+1-i)(p-1)+1-j +u-j'-i =  (u+1-i)p -j -j' =  (u+1+j'-i)p -j'',$$
which is prime to $p$. Thus each wide staircase satisfies a stronger
condition than Fayer's criterion.

 The dimension can again be calculated by means of the
hook formula. The first $p-1$ boxes in the first row have hook
numbers
$$ pu -1, pu-2, \dots, pu-(p-1),$$
whose sum is $(p-1)pu - \binom p2 =\binom p2  (2u-1).$

The next $p-1$ boxes in the first row have hook numbers
$$  p(u-1) -1, p(u-1)-2, \dots,  p(u-1)-(p-1),$$
whose sum is $(p-1)p(u-1) - \binom p2 = \binom p2 (2u-3).$

Thus the sum of the hook numbers in the first row is
$$ \binom p2((2u-1) +(2u-3) + \cdots + 1)  =  \binom p2 u^2.$$

Summing over all rows yields
$$\sum h_x = \binom p2 \sum _{k=1}^u k^2 = \binom p2 \frac{u(u+1)(2u+1)}6 = \binom p2 \frac{(2u+1)n}3.$$

\end{example}

\begin{lem}\label{hook51}
For  any integer $u$,  let $\mu $ be the wide staircase $T_{u}$ of
$u$ rows. Let $n = (p-1){ u \choose 2}$. Then
\[\left(\frac{6n}{p(p-1)(2u+1)}
\right)^n\cdot\left(\frac{1}{e} \right)^n <f^\mu  \qquad
(\mbox{where} \quad e=2.718281828\ldots).\] In particular, if
$\alpha\le \frac{3n}{(2u+1)e},$ then $\al^n\le f^\mu$.
\end{lem}
\begin{proof}
We imitate the proof of Lemma~\ref{hook5}.
Since the geometric mean is bounded by the arithmetic mean,
\[\left(\prod_{x\in\mu }h_x \right )^{1/n}\le\frac{1}{n}\sum_{x\in\mu }h_x
\le \binom p2 \frac{(2u+1)n}3 = \frac{p(p-1)(2u+1)n}6,\] in view of
Example~\ref{rect1},
%
together  with $({n}/e)^{n}<{n}!$, implies that
\[
\left(\frac{6n}{p(p-1)(2u+1)} \right)^{n}\cdot\left(\frac{1}{e}
\right)^{n}=\left(\frac{n}{e}
\right)^{n}\cdot\left(\frac{6}{p(p-1)(2u+1)} \right)^{n}
<\frac{n!}{\prod_{x\in\mu }h_x}=f^\mu .\]
\end{proof}

%
%
%
%
%
%
%

%

\begin{lem}\label{hook61}
Let $A$ be a PI~algebra over a field of characteristic $p,$ that
satisfies an identity of degree $d$. Choose a natural number $u$
such that, for $n = (p-1)\binom{u+1}2,$
$$\frac{6n}{p(p-1)(2u+1)}\cdot\frac{1}{e}\ge (d-1)^2 .$$
 Let $\lm\vdash n$ be any
partition of $n$ corresponding to the ``wide staircase''~$T_{u}$.
Then the elements of the corresponding $F[S_n]$-bimodule
$I_\lm\subseteq V'_n$ are sparse identities of $A$.
\end{lem}

\begin{proof}
By Remark~\ref{regest},

\[c_{2n}(A)\le (d-1)^{4n}\le\left(\frac{6n}{p(p-1)(2u+1)}\cdot\frac{1}{e}\right)^n<
s^\lm,\] and we conclude from Lemma~\ref{hook45}.
\end{proof}

\subsubsection{Existence of Capelli identities}$ $

We are ready for a version of Proposition~\ref{strong.2}.


\begin{thm}\label{sparse}~\cite{kemerp}  Any
 PI- algebra  $A$ over a field $ F $ of characteristic $p>0$  satisfies a   Capelli  identity.
Explicitly:

\medskip
(a) ~Suppose the $F$-algebra $A$ satisfies an identity of degree
$d$. Then $A$ satisfies a sparse identity of degree $ d'=
(p-1)p\binom{u+1}2,$ where $ \frac{3u(u+1)}{p(2u+1)} \ge (d-1)^2 e.
$

\medskip
(b) ~Suppose $A=F\{a_1,\ldots,a_r\}$, and  $A$ satisfies an identity
of degree $d$ and take $d'$ as in~(a). Let $n= r^{d'}+d' \approx
r^{4e^2(d-1)^4}$. Then $A$ satisfies the Capelli identity $\Capl_n$.
\end{thm}
\begin{proof}  (a) is by Lemma~\ref{hook61}. Then (b) follows from
Theorem~\ref{kemer.capelli}.\end{proof}

For example, since
 $\frac{ u+1 }{2u+1} \ge \frac 12,$ we could
take $u \ge \frac {2pe (d-1)^2}3 . $
%
%

This concludes the proof of Theorem~\ref{capelli.kemerp} in the
affine case.

\section{Results and proofs over  Noetherian base rings}
\label{Noeth}

We turn to the case where $C$ is a commutative  Noetherian ring.
In general, we say a $C$-algebra is PI if it satisfies a
polynomial identity having at least one coefficient equal to 1.
Let us indicate the modifications that need to be made in order to
obtain proofs of Theorems~\ref{Braun} and \ref{weaknul}.

The method of proof of Theorem~\ref{nilrep1}(2) (for the case in
which the base ring $C$ is a field) was to verify the ``weak
Nullstellensatz'', and a similar proof works for $A$ commutative
when $C$ is Jacobson, cf.~\cite[Proposition 4.4.1]{rowen1}. Thus
we have Theorems~\ref{Braun} and ~\ref{weaknul} in the commutative
case, which provide the base for our induction to prove
Theorem~\ref{thm.raz1}. The argument is carried out using
Zubrilin's methods (which were given over an arbitrary commutative
base ring.)



It remains to find a way of proving Kemer's Capelli Theorem  over
arbitrary Noetherian base rings. One could do this directly using
Young diagrams, but there also is a ring-theoretic reduction. The
following observations about Capelli identities are useful.

                            \begin{lem}\label{Capprod}

                            (i) Suppose $n = n_1 n_2 \cdots n_t.$  If $A$ satisfies the identity $\Capl_{n_1} \times
                            \dots \times \Capl_{n_t},$
                            then   $A$ satisfies the  Capelli identity $\Capl_n$.

                            (ii) If $I \triangleleft A$  and $A/I$   satisfies $\Capl_m$
                            for $m$ odd, with $I^k = 0,$ then $A$ satisfies $\Capl_{km}.$

                           (iii) If $I \triangleleft A$  and $A/I$   satisfies $\Capl_m$
                            with $I^k = 0,$ then $A$ satisfies $\Capl_{k(m+1)}.$ \end{lem}
                            \begin{proof}
                            (i) Viewing the symmetric group $S_{n_1}\times \dots \times
                            S_{n_m} \hookrightarrow S_n,$ we partition $S_n$ into orbits under
                            the subgroup  $S_{n_1}\times \dots \times
                            S_{n_m}$ and
                            match the permutations in
                            $\Capl_n$.

                            (ii) This time we note that any interchange of two odd-order
                            sets of letters has negative sign, so we partition $S_{km}$ into
                            $k$ parts each with $m$ letters.

                            (iii) Any algebra satisfying
                              $\Capl_m$
                            for $m$ even, also satisfies
                            $\Capl_{m+1}$, and $m+1$ is odd.                       \end{proof}

Thus, it suffices
 to prove that $A$ satisfies a product of  Capelli identities.

\begin{thm} \label{kemerfirst9}
Any affine PI algebra over a commutative Noetherian base ring $C$
satisfies some Capelli identity.
\end{thm}
\begin{proof} By Noetherian induction, we may assume that the theorem
holds for every affine PI-algebra over a proper homomorphic image
of $C$.

First we do do the case where $C$ is an integral domain, and $A =
C\{a_1, \dots, a_\ell\}$ satisfies some multilinear PI $f$. It is
enough to assume that $A$ is the relatively free algebra $C\{x_1,
\dots, x_n\}/I$ (where $I$ is the T-ideal generated by $f$). Let $F$
be the field of fractions of $C$. Then $A_F: = A \otimes _C F$ is
also a PI-algebra, and thus, by Theorem~\ref{capelli.kemerp}
satisfies some Capelli identity $f_1 = \Capl_{n}.$ Thus the image
$\bar f_1$ of $f_1$ in $A$ becomes~0 when we tensor by $F$, which
means that there is some $s\in C$ for which $s  f_1 = 0.$ Letting
$I'$ denote the T-ideal of $A$ generated by the image of $f_1,$ we
see that $s I' = 0.$ If $s =1$ then we are done,  so we may assume
that $s\in C$ is not invertible. Then $A/s A$ is an affine
PI-algebra over the proper homomorphic image $C/sC$ of $C$, and by
Noetherian induction, satisfies some Capelli identity $\Capl_{m}$,
so $A/(sA \cap  I')$ satisfies~$\Capl_{\max\{m,n\}}$. But $sA \cap
I'$ is nilpotent modulo $sAI' = As I'
 = 0,$ implying by Lemma~\ref{Capprod} that
$A$ satisfies some Capelli identity.

For the general case,
 the nilpotent
radical $N$ of $C$ is a finite intersection $P_1 \cap \dots \cap
P_t$ of prime ideals. By the previous paragraph, $A/P_jA$, being
an affine PI-algebra over the integral domain~$C/P_j,$ satisfies a
suitable Capelli identity $\Capl_{n_j},$ for $1 \le j \le t,$ so
$A/\cap (P_jA)$ satisfies $\Capl_{n},$ where $n = \max \{ n_1,
\dots, n_t\}.$ But $\cap (P_jA)$ is nilpotent modulo $NA$, so, by
Lemma~\ref{Capprod}, $A/NA$ satisfies a suitable Capelli identity
$\Capl_{n}.$ Furthermore, $N^m = 0$ for some $m$, implying again
 by
Lemma~\ref{Capprod} that $A$ satisfies $C^{mn}.$
\end{proof}

%
%
%
%


\begin{thebibliography}{99}



\bibitem{amitsur}  S.A.~Amitsur, A generalization of Hilbert Nullstellensatz, Proc.~Amer.~Math.~Soc.~{\bf 8} (1957) 649-656.

\bibitem{amitsur1} S.A.~Amitsur, A note on P.I. rings, Israel J. Math.
{\bf 10} (1971) 210--211.

%

\bibitem{amitsur.procesi} S.A.~Amitsur and C.~Procesi,~~Jacobson rings and Hilbert algebras with polynomial identities,
Ann.~Mat.~Pura Appl.~{\bf 71} (1966) 67--72.

\bibitem{amitsur.regev} S.A.~Amitsur and A.~Regev:~~P.I. algebras and their
cocharacters,  J. of Algebra {\bf 78} (1982) 248--254.

\bibitem{amitsur.small} S.A.~Amitsur and L.~Small,~~{\it Affine algebras with polynomial
identities,}  Supplemento ai Rendiconti del Circolo Matematico di
Palermo {\bf 31} (1993).

\bibitem{BBRY}  A.~Belov,  L.~Bokut, L.H.~Rowen, and   J.T. Yu,
{\it The Jacobian Conjecture, together with Specht and Burnside-type
problems}, Proc. Groups of Automorphisms in Birational and Affine
Geometry, Springer, editors M.Zaidenberg, M. Rich, and M.~ Reizakis,
to appear.






\bibitem{BR} A.K.~Belov and L.H.~Rowen, {\it Computational aspects of Polynomial
Identities,}
A. K. Peters (2005).
%








%
%
%
%



\bibitem{braun}   A.~Braun, The nilpotency of the radical in a finitely generated PI-ring,
J. Algebra {\bf 89} (1984), 375-396.






































\bibitem{F} M.~Fayers, Irreducible Specht modules for Hecke algebras of type A,
Advances in Math. {\bf 193} (2005), 438--452.

\bibitem{FLM} M.~Fayers, S.~Lyle, S.~Martin, p-restriction of
partitions and homomorphisms between Specht modules, J.~Algebra {\bf
306} (2006), 175--190.

















\bibitem{jacobson} N.~Jacobson, {\it Basic Algebra II}, second
edition, Freeman and company (1989).

%


\bibitem{james2} G.~D.~James, {\it The Representation Theory of the
Symmetric Groups}, Lecture Notes in Math, Vol.~682,
Springer--Verlag, New York, NY, (1978).

\bibitem{jamesmathas} G.~James and A.~Mathas, {\it The irreducible Specht modules in
characteristic 2}, Bull. London Math. Soc.~{\bf 31} (1999), 457–-62.

\bibitem{jameskerber} G.~D.~James and A.~Kerber,  {\it The Representation Theory of the
Symmetric group}, Encyclopedia of Mathematics and its
Applications, Vol.~16, Addison--Wesley, Reading, MA, (1981).


\bibitem{kemer.0.5} A.R~ Kemer, Capelli identities and nilpotence of the radical of a finitely generated
PI-algebra, Dokl. Akad. Nauk SSSR {\bf 255} (1980), 793-797 (Russian). English translation: Soviet
Math. Dokl. {\bf 22} (1980), 750-753.

\bibitem{kemer} A.R.~Kemer, Ideals of identities of associative
algebras, Amer. Math. Soc. Translations of monographs {\bf 87}
(1991).

\bibitem{kemerp} A.R.~Kemer,
    Multilinear identities of the algebras over
a field of characteristic $p$, Internat. J. Algebra Comput.
\textbf{5}  no.~2, (1995), 189--197.













\bibitem{lew} J.~Lewin, A matrix representation
for associative algebras ~I and II, Trans.~Amer.~Math. Soc.~
\textbf{188}(2), 293--317 (1974).

\bibitem{lvov} L'vov, Unpublished (Russian).

\bibitem{macdonald} I.G.~Macdonald, {\it Symmetric Functions and Hall
Polynomials}, 2nd edition, Oxford University Press, Oxford,
(1995).

\bibitem{mathas} A. Mathas, {\it Iwahori–-Hecke algebras and Schur algebras of the
symmetric group}, University Lecture Series~15, American
Mathematical Society, Providence, RI (1999).













%
%

\bibitem{razmyslov3} Yu.P.~Razmyslov, The Jacobson radical in PI-algebras (Russian),
Algebra i Logika {\bf 13} (1974), 337-360. English translation: Algebra and Logika
{\bf 13} (1974), 192-204.


\bibitem{regev1} A.~Regev, Existence of identities in $A\otimes B$, Israel
J. Math. {\bf 11} (1972), 131--152.


\bibitem{regev5} A.~Regev, The representations of $S_n$ and explicit identities
for P.I.~algebras,~~  J. Algebra ~{\bf 51} (1978), 25--40.

\bibitem{RSS}
  Richard Resco, Lance W. Small and J. T. Stafford, Krull and Global Dimensions of Semiprime Noetherian $PI$-Rings
Transactions of the American Mathematical Society~{\bf 274, No. 1}
(1982), 285--295












\bibitem{rowen1} L.H.~Rowen, {\it Polynomial Identities in Ring Theory}, Pure and Applied
Mathematics, 84. Academic Press, Inc. [Harcourt Brace Jovanovich,
Publishers], New York-London, (1980).

%

\bibitem{rowen3} L.H.~Rowen, {\it Ring Theory}, Vol.~II. Pure and Applied Mathematics, 128.
Academic Press, Inc., Boston, MA, (1988).


\bibitem{rowen4} L.H.~Rowen, {\it {Graduate algebra: Commutative
View}}, AMS Graduate Studies in Mathematics~\textbf{73}, 2006.


\bibitem{rowen5} L.H.~Rowen, {\it {Graduate algebra: Noncommutative
View}}, AMS Graduate Studies in Mathematics~\textbf{91}, 2008.



\bibitem{sagan} B.E.~Sagan, {\it The Symmetric Group: Representations, Combinatorial
Algorithms, and Symmetric Functions}, 2nd edition, Graduate Texts in
Mathematics 203, Springer-Verlag (2000).






\bibitem{shirshov1} A.I.~Shirshov, On certain non associative nil
rings and algebraic algebras (Russian), Mat. Sb. {\bf 41} (1957), 381-394.

\bibitem{shirshov2} A.I.~Shirshov, On rings with identity relations
(Russian), Mat. Sb. {\bf 43} (1957), 277-283.

 \bibitem{sm}
Small, L.W.,
\newblock {\em An example in PI rings},
\newblock  J. Algebra {\bf 17} (1971), 434--436.

\bibitem{zubrilin.1} K.A.~Zubrilin, Algebras
satisfying Capelli identities, Sbornic Math. {\bf 186} no. 3
(1995) 359-370.
















\end{thebibliography}
\end{document}